\documentclass[a4paper,leqno,12pt]{amsart}
\usepackage{amsmath,amsthm}
\usepackage{amssymb}
\usepackage{xspace}
\usepackage{bm}
\usepackage{amstext}
\usepackage{amsfonts}
\usepackage{graphicx}
\usepackage[mathscr]{euscript}
\usepackage{amscd}
\usepackage{latexsym}
\usepackage{amssymb}
\usepackage{enumerate}
\usepackage{mathrsfs}
\usepackage[cmtip,all]{xy}
\usepackage{tikz}
\usetikzlibrary{shapes.geometric}
\begin{document}
%
\theoremstyle{plain}
\swapnumbers
    \newtheorem{thm}[figure]{Theorem}
    \newtheorem{prop}[figure]{Proposition}
    \newtheorem{lemma}[figure]{Lemma}
    \newtheorem{keylemma}[figure]{Key Lemma}
    \newtheorem{corollary}[figure]{Corollary}
    \newtheorem{fact}[figure]{Fact}
    \newtheorem{subsec}[figure]{}
    \newtheorem*{propa}{Proposition A}
    \newtheorem*{thma}{Theorem A}
    \newtheorem*{thmb}{Theorem B}
    \newtheorem*{thmc}{Theorem C}
    \newtheorem*{thmd}{Theorem D}
\theoremstyle{definition}
    \newtheorem{defn}[figure]{Definition}
    \newtheorem{examples}[figure]{Examples}
    \newtheorem{notn}[figure]{Notation}
    \newtheorem{summary}[figure]{Summary}
\theoremstyle{remark}
        \newtheorem{remark}[figure]{Remark}
        \newtheorem{remarks}[figure]{Remarks}
        \newtheorem{example}[figure]{Example}
        \newtheorem{warning}[figure]{Warning}
    \newtheorem{assume}[figure]{Assumption}
    \newtheorem{ack}[figure]{Acknowledgements}
\renewcommand{\thefigure}{\arabic{section}.\arabic{figure}}
%
%
%
\newenvironment{myeq}[1][]
{\stepcounter{figure}\begin{equation}\tag{\thefigure}{#1}}
{\end{equation}}
\newcommand{\myeqn}[2][]
{\stepcounter{figure}\begin{equation}
     \tag{\thefigure}{#1}\vcenter{#2}\end{equation}}
\newcommand{\mydiag}[2][]{\myeq[#1]\xymatrix{#2}}
\newcommand{\mydiagram}[2][]
{\stepcounter{figure}\begin{equation}
     \tag{\thefigure}{#1}\vcenter{\xymatrix{#2}}\end{equation}}
\newcommand{\mysdiag}[2][]
{\stepcounter{figure}\begin{equation}
     \tag{\thefigure}{#1}\vcenter{\xymatrix@R=10pt@C=20pt{#2}}\end{equation}}
\newcommand{\mytdiag}[2][]
{\stepcounter{figure}\begin{equation}
     \tag{\thefigure}{#1}\vcenter{\xymatrix@R=15pt@C=25pt{#2}}\end{equation}}
\newcommand{\myudiag}[2][]
{\stepcounter{figure}\begin{equation}
     \tag{\thefigure}{#1}\vcenter{\xymatrix@R=15pt@C=20pt{#2}}\end{equation}}
\newcommand{\myrdiag}[2][]
{\stepcounter{figure}\begin{equation}
     \tag{\thefigure}{#1}\vcenter{\xymatrix@R=20pt@C=15pt{#2}}\end{equation}}
\newcommand{\myqdiag}[2][]
{\stepcounter{figure}\begin{equation}
     \tag{\thefigure}{#1}\vcenter{\xymatrix@R=20pt@C=22pt{#2}}\end{equation}}
\newcommand{\myfigure}[2][]
{\stepcounter{figure}\begin{equation}
     \tag{\thefigure}{#1}\vcenter{#2}\end{equation}}
\newcommand{\mywdiag}[2][]
{\stepcounter{figure}\begin{equation}
     \tag{\thefigure}{#1}\vcenter{\xymatrix@R=20pt@C=12pt{#2}}\end{equation}}
\newcommand{\myzdiag}[2][]
{\stepcounter{figure}\begin{equation}
     \tag{\thethm}{#1}\vcenter{\xymatrix@R=5pt@C=20pt{#2}}\end{equation}}
%
\newenvironment{mysubsection}[2][]
{\begin{subsec}\begin{upshape}\begin{bfseries}{#2.}
\end{bfseries}{#1}}
{\end{upshape}\end{subsec}}
\newenvironment{mysubsect}[2][]
{\begin{subsec}\begin{upshape}\begin{bfseries}{#2\vsn.}
\end{bfseries}{#1}}
{\end{upshape}\end{subsec}}
\newcommand{\supsect}[2]
{\vspace*{-5mm}\quad\\\begin{center}\textbf{{#1}}\vsm.~~~~\textbf{{#2}}\end{center}}
\newcommand{\sect}{\setcounter{figure}{0}\section}
%
%
\newcommand{\wh}{\ -- \ }
\newcommand{\wwh}{-- \ }
\newcommand{\w}[2][ ]{\ \ensuremath{#2}{#1}\ }
\newcommand{\ww}[1]{\ \ensuremath{#1}}
\newcommand{\www}[2][ ]{\ensuremath{#2}{#1}\ }
\newcommand{\wwb}[1]{\ \ensuremath{(#1)}-}
\newcommand{\wb}[2][ ]{\ (\ensuremath{#2}){#1}\ }
\newcommand{\wref}[2][ ]{\ (\ref{#2}){#1}\ }
\newcommand{\wwref}[2]{\ (\ref{#1})-(\ref{#2})\ }
%
%
\newcommand{\hsp}{\hspace*{9 mm}}
\newcommand{\hs}{\hspace*{5 mm}}
\newcommand{\hsn}{\hspace*{1 mm}}
\newcommand{\hsm}{\hspace*{2 mm}}
\newcommand{\vsn}{\vspace{2 mm}}
\newcommand{\vs}{\vspace{5 mm}}
\newcommand{\vsm}{\vspace{3 mm}}
\newcommand{\vsp}{\vspace{9 mm}}
%
%
\newcommand{\hra}{\hookrightarrow}
\newcommand{\xra}[1]{\xrightarrow{#1}}
\newcommand{\xepic}[1]{\xrightarrow{#1}\hspace{-5 mm}\to}
\newcommand{\lora}{\longrightarrow}
\newcommand{\lra}[1]{\langle{#1}\rangle}
\newcommand{\llrra}[1]{\langle\langle{#1}\rangle\rangle}
\newcommand{\llrr}[2]{\llrra{#1}\sb{#2}}
\newcommand{\llrrp}[2]{\llrra{#1}'\sb{#2}}
\newcommand{\lrf}{\langle\langle f\lo{0,1}\rangle\rangle}
\newcommand{\lrfn}[1]{\lrf\sb{#1}}
\newcommand{\lras}[1]{\langle{#1}\rangle\sb{\ast}}
\newcommand{\lrau}[1]{\langle{#1}\rangle\sp{\ast}}
\newcommand{\vlam}{\vec{\lambda}}
\newcommand{\epic}{\to\hspace{-3.5 mm}\to}
\newcommand{\xhra}[1]{\overset{#1}{\hookrightarrow}}
\newcommand{\efp}{\to\hspace{-1.5 mm}\rule{0.1mm}{2.2mm}\hspace{1.2mm}}
\newcommand{\efpic}{\mbox{$\to\hspace{-3.5 mm}\efp$}}
\newcommand{\up}[1]{\sp{(#1)}}
\newcommand{\bup}[1]{\sp{[{#1}]}}
\newcommand{\lo}[1]{\sb{(#1)}}
\newcommand{\lolr}[1]{\sb{\lra{#1}}}
\newcommand{\bp}[1]{\sb{[#1]}}
\newcommand{\hfsm}[2]{{#1}\ltimes{#2}}
\newcommand{\sms}[2]{{#1}\wedge{#2}}
\newcommand{\rest}[1]{\lvert\sb{#1}}
\newcommand{\adj}[2]{\substack{{#1}\\ \rightleftharpoons \\ {#2}}}
%
%
\newcommand{\ab}{\operatorname{ab}}
\newcommand{\Ab}{\operatorname{Ab}}
\newcommand{\Aut}{\operatorname{Aut}}
\newcommand{\Cf}{\C\sb{\operatorname{f}}}
\newcommand{\Cof}[1]{\operatorname{Cof}(#1)}
\newcommand{\Coker}{\operatorname{Coker}}
\newcommand{\colim}{\operatorname{colim}}
\newcommand{\colimit}[1]
{\raisebox{-1.7ex}{$\stackrel{\textstyle\colim}{\scriptstyle{#1}}$}}
\newcommand{\comp}{\mbox{\sf comp}}
\newcommand{\Cone}{\operatorname{Cone}}
\newcommand{\csk}[1]{\operatorname{csk}\sp{#1}}
\newcommand{\diag}{\operatorname{diag}}
\newcommand{\ev}{\operatorname{ev}}
\newcommand{\exc}{\operatorname{ex}}
\newcommand{\Ext}{\operatorname{Ext}}
\newcommand{\Fib}{\operatorname{Fib}}
\newcommand{\fin}{\operatorname{fin}}
\newcommand{\fwd}{\operatorname{fwd}}
\newcommand{\gr}{\operatorname{gr}}
\newcommand{\ho}{\operatorname{ho}}
\newcommand{\hocofib}{\operatorname{hocofib}}
\newcommand{\hocolim}{\operatorname{hocolim}}
\newcommand{\holim}{\operatorname{holim}}
\newcommand{\Hom}{\operatorname{Hom}}
\newcommand{\inc}{\operatorname{inc}}
\newcommand{\Id}{\operatorname{Id}}
\newcommand{\Image}{\operatorname{Im}}
\newcommand{\Ker}{\operatorname{Ker}}
\newcommand{\md}{\operatorname{mid}}
\newcommand{\Ob}{\operatorname{Obj}}
\newcommand{\op}{\sp{\operatorname{op}}}
\newcommand{\pt}{\operatorname{pt}}
\newcommand{\red}{\operatorname{red}}
\newcommand{\sgn}[1]{\operatorname{sgn}(#1)}
\newcommand{\sk}[1]{\operatorname{sk}\sb{#1}}
\newcommand{\SL}[1]{\operatorname{SL}\sb{#1}(\ZZ)}
\newcommand{\Sq}[1]{\operatorname{Sq}\sp{#1}}
\newcommand{\Tot}{\operatorname{Tot}}
\newcommand{\uTot}{\underline{\Tot}}
%
%
\newcommand{\map}{\operatorname{map}}
\newcommand{\mapa}{\map\sb{\ast}}
%
%
\newcommand{\Hu}[3]{H\sp{#1}({#2};{#3})}
\newcommand{\Hus}[2]{\Hu{\ast}{#1}{#2}}
\newcommand{\HuF}[2]{\Hu{#1}{#2}{\Fp}}
\newcommand{\HuFs}[1]{\HuF{\ast}{#1}}
\newcommand{\HuT}[2]{\Hu{#1}{#2}{\Ft}}
\newcommand{\HuTs}[1]{\HuT{\ast}{#1}}
%
%
\newcommand{\Hi}[3]{H\sb{#1}({#2};{#3})}
\newcommand{\His}[2]{\Hi{\ast}{#1}{#2}}
\newcommand{\HiF}[2]{\Hi{#1}{#2}{\Fp}}
\newcommand{\HiZ}[2]{\Hi{#1}{#2}{\ZZ}}
\newcommand{\HiZp}[2]{\Hi{#1}{#2}{\Zp}}
\newcommand{\HiQ}[2]{\Hi{#1}{#2}{\QQ}}
\newcommand{\HiFs}[1]{\His{#1}{\Fp}}
\newcommand{\HiT}[2]{\Hi{#1}{#2}{\Ft}}
\newcommand{\HiTs}[1]{\HiT{\ast}{#1}}
%
%
\newcommand{\Ei}[3]{E\sb{#1}\sp{{#2},{#3}}}
\newcommand{\Eis}[2]{E\sb{#1}(#2)}
\newcommand{\Eu}[3]{E\sp{#1}\sb{{#2},{#3}}}
\newcommand{\Eus}[1]{E\sp{#1}}
\newcommand{\Euz}[1]{E\sp{0}({#1})}
\newcommand{\Eot}[2]{\Eu{1}{#1}{#2}}
\newcommand{\Ett}[2]{\Eu{2}{#1}{#2}}
%
%
\newcommand{\A}{\mathcal{A}}
\newcommand{\tA}{\widetilde{A}}
\newcommand{\wA}{\widehat{A}}
\newcommand{\B}{\mathcal{B}}
\newcommand{\wB}{\widehat{B}}
\newcommand{\C}{\mathcal{C}}
\newcommand{\wC}{\widehat{C}}
\newcommand{\hC}{\hat{\C}}
\newcommand{\D}{\mathcal{D}}
\newcommand{\wD}{\widehat{D}}
\newcommand{\tD}{\widetilde{D}}
\newcommand{\E}{\mathcal{E}}
\newcommand{\wE}{\widehat{E}}
\newcommand{\cF}[1]{\mathcal{F}\sp{#1}}
\newcommand{\Fc}[1]{\mathcal{F}\sb{#1}}
\newcommand{\wf}{\widehat{f}}
\newcommand{\G}{\mathcal{G}}
\newcommand{\cII}{\mathcal{I}}
\newcommand{\cI}[1]{\cII\sp{#1}}
\newcommand{\hI}{\widehat{I}}
\newcommand{\hhI}{\widehat{\hI}}
\newcommand{\uI}{\underline{I}}
\newcommand{\wI}[1]{\hI\sp{#1}}
\newcommand{\hJ}{\widehat{J}}
\newcommand{\hhJ}{\widehat{\hJ}}
\newcommand{\uJ}{\underline{J}}
\newcommand{\hK}{\widehat{K}}
\newcommand{\uK}{\underline{K}}
\newcommand{\hhK}{\widehat{\hK}}
\newcommand{\KK}[1]{{\mathcal{K}}\sb{#1}}
\newcommand{\hL}{\widehat{L}}
\newcommand{\hhL}{\widehat{\hL}}
\newcommand{\LL}{{\mathcal L}}
\newcommand{\M}{\mathcal{M}}
\newcommand{\hM}{\widehat{M}}
\newcommand{\hhM}{\widehat{\hM}}
\newcommand{\Map}{{\EuScript Map}}
\newcommand{\MA}{\Map\sb{\A}}
\newcommand{\MAB}{\Map\sb{\A}\sp{\B}}
\newcommand{\MB}{\Map\sp{\B}}
\newcommand{\hN}{\widehat{N}}
\newcommand{\hhN}{\widehat{\hN}}
\newcommand{\OO}{\mathcal{O}}
\newcommand{\Pc}{\mathcal{P}}
\newcommand{\cP}[1]{\Pc\sp{#1}}
\newcommand{\hP}[2]{\widehat{\mathcal{P}}\sp{#1}\sb{#2}}
\newcommand{\Ss}{\mathcal{S}}
\newcommand{\Sa}{\Ss\sb{\ast}}
\newcommand{\Sr}{\Ss\sp{\red}}
\newcommand{\U}{\mathcal{U}}
\newcommand{\eW}{{\EuScript W}}
\newcommand{\eX}{{\EuScript X}}
\newcommand{\eY}{{\EuScript Y}}
\newcommand{\eZ}{{\EuScript Z}}
%
%
\newcommand{\hy}[2]{{#1}\text{-}{#2}}
\newcommand{\Alg}[1]{{#1}\text{-}{\mbox{\sf Alg}}}
\newcommand{\Pa}[1][ ]{$\Pi$-algebra{#1}}
\newcommand{\PAlg}{\Alg{\Pi}}
\newcommand{\pis}{\pi\sb{\ast}}
\newcommand{\gS}[1]{{\EuScript S}\sp{#1}}
\newcommand{\Mod}[1]{{#1}\text{-}{\mbox{\sf Mod}}}
\newcommand{\us}{u\sp{\ast}}
\newcommand{\Set}{\mbox{\sf Set}}
\newcommand{\Seta}{\Set\sb{\ast}}
\newcommand{\Cat}{\mbox{\sf Cat}}
\newcommand{\Ch}{\mbox{\sf Ch}}
\newcommand{\DK}{\mbox{\sf DK}}
\newcommand{\Grp}{\mbox{\sf Gp}}
\newcommand{\OC}{\hy{\OO}{\Cat}}
\newcommand{\SC}{\hy{\Ss}{\Cat}}
\newcommand{\SaC}{\hy{\Sa}{\Cat}}
\newcommand{\SO}{(\Ss,\OO)}
\newcommand{\SaO}{(\Sa,\OO)}
\newcommand{\SOC}{\hy{\SO}{\Cat}}
\newcommand{\SaOC}{\hy{\SaO}{\Cat}}
\newcommand{\Top}{\mbox{\sf Top}}
\newcommand{\Tz}{\Top\sb{0}}
\newcommand{\Topa}{\Top\sb{\ast}}
%
%
\newcommand{\CC}{\mathbb C}
\newcommand{\CP}[1]{\CC\mathbf{P}\sp{#1}}
\newcommand{\FF}{\mathbb F}
\newcommand{\Fp}{\FF\sb{p}}
\newcommand{\Ft}{\FF\sb{2}}
\newcommand{\Fq}{\FF\sb{q}}
\newcommand{\II}{\mathbb I}
\newcommand{\NN}{\mathbb N}
\newcommand{\QQ}{\mathbb Q}
\newcommand{\RR}{\mathbb R}
\newcommand{\ZZ}{\mathbb Z}
\newcommand{\Zp}{\ZZ\lo{p}}
\newcommand{\Zt}{\ZZ\lo{2}}
%
%
\newcommand{\Del}{\mathbf{\Delta}}
\newcommand{\Deln}[1]{\Del\sp{#1}}
\newcommand{\Delnk}[2]{\Deln{#1}\lo{#2}}
\newcommand{\Dop}{\Delta\op}
\newcommand{\res}{\operatorname{res}}
\newcommand{\Dres}{\Delta\sb{\res}}
\newcommand{\Drop}{\Dres\op}
\newcommand{\Dp}{\Delta\sb{+}}
\newcommand{\Dresp}{\Delta\sb{\res,+}}
\newcommand{\Drn}[1]{\Delta\lolr{#1}}
\newcommand{\Drnop}[1]{\Drn{#1}\op}
\newcommand{\Du}{\Del\sp{\bullet}}
%
%
\newcommand{\bA}{{\mathbf A}}
\newcommand{\bB}{{\mathbf B}}
\newcommand{\bC}{{\mathbf C}}
\newcommand{\bD}{{\mathbf D}}
\newcommand{\cD}{D}
\newcommand{\bE}{{\mathbf E}}
\newcommand{\be}[1]{{\mathbf e}\sp{#1}}
\newcommand{\bF}{{\mathbf F}}
\newcommand{\Fv}[1]{\bF\bp{#1}}
\newcommand{\Fk}[1]{F\sp{#1}}
\newcommand{\Fpp}{\hspace{0.5mm}'\hspace{-0.7mm}F}
\newcommand{\Fpk}[1]{\hspace{0.5mm}'\hspace{-0.7mm}F\sp{#1}}
\newcommand{\Fn}[2]{F\sp{#1}\bp{#2}}
\newcommand{\bG}{{\mathbf G}}
\newcommand{\Gv}[1]{\bG\bp{#1}}
\newcommand{\hGv}[1]{\widehat{\bG}\bp{#1}}
\newcommand{\Gn}[2]{G\sp{#1}\bp{#2}}
\newcommand{\hGn}[2]{\widehat{G}\sp{#1}\bp{#2}}
\newcommand{\tg}[1]{\widetilde{g}\sp{#1}}
\newcommand{\bH}{{\mathbf H}}
\newcommand{\wH}{\widehat{H}}
\newcommand{\Hv}[1]{\bH\bp{#1}}
\newcommand{\hHv}[1]{\wH\bp{#1}}
\newcommand{\Hn}[2]{H\sp{#1}\bp{#2}}
\newcommand{\tH}[1]{\widetilde{H}\sp{#1}}
\newcommand{\tHn}[2]{\tH{#1}\bp{#2}}
\newcommand{\bi}{{\mathbf i}}
\newcommand{\bj}{{\mathbf j}}
\newcommand{\bK}{{\mathbf K}}
\newcommand{\bL}{{\mathbf L}}
\newcommand{\KP}[2]{\bK({#1},{#2})}
\newcommand{\KR}[1]{\KP{R}{#1}}
\newcommand{\KZ}[1]{\KP{\ZZ}{#1}}
\newcommand{\KZp}[1]{\KP{\Zp}{#1}}
\newcommand{\KZt}[1]{\KP{\Zt}{#1}}
\newcommand{\KF}[1]{\KP{\Fp}{#1}}
\newcommand{\bM}[1]{{\mathbf M}\sp{#1}}
\newcommand{\bP}{{\mathbf P}}
\newcommand{\bS}[1]{{\mathbf S}\sp{#1}}
\newcommand{\bT}{\mathbf{\Theta}}
\newcommand{\TA}{\bT\sp{\A}}
\newcommand{\TB}{\bT\sb{\B}}
\newcommand{\ThB}{\Theta\sb{\B}}
\newcommand{\TAB}{\bT\sp{\A}\sb{\B}}
\newcommand{\TR}{\Theta\sb{R}}
\newcommand{\TRl}{\Theta\sb{R}\sp{\lambda}}
\newcommand{\bU}{{\mathbf U}}
\newcommand{\bV}{{\mathbf V}}
\newcommand{\bW}{{\mathbf W}\!}
\newcommand{\bX}{{\mathbf X}}
\newcommand{\hX}{\widehat{\bX}}
\newcommand{\tX}{\widetilde{\bX}}
\newcommand{\bY}{{\mathbf Y}}
\newcommand{\hY}{\widehat{\bY}}
\newcommand{\bZ}{{\mathbf Z}}
%
%
\newcommand{\bdel}{\bar{\delta}}
\newcommand{\eps}[1]{\epsilon\sb{#1}}
\newcommand{\gam}[1]{\gamma\sb{#1}}
\newcommand{\iot}[1]{\iota\sb{#1}}
\newcommand{\vare}{\varepsilon}
\newcommand{\var}[1]{\vare\bup{#1}}
\newcommand{\tvar}[1]{\widetilde{\vare}\bup{#1}}
\newcommand{\ett}[1]{\eta\sb{#1}}
\newcommand{\veta}{\vec{\eta}}
\newcommand{\blam}{\bar{\lambda}}
\newcommand{\wvarp}{\widehat{\varphi}}
\newcommand{\wrho}{\widehat{\rho  }}
\newcommand{\tS}{\widetilde{\Sigma}}
\newcommand{\wsig}{\widehat{\sigma}}
\newcommand{\prn}[1]{\iota\bup{#1}}
\newcommand{\bv}{{\bm \vare}}
\newcommand{\bve}[1]{\bv\bup{#1}}
\newcommand{\hve}[1]{\widehat{\bv}\bup{#1}}
\newcommand{\tve}[1]{\widetilde{\bv}\bup{#1}}
%
%
\newcommand{\bd}{\mathbf{d}\sb{0}}
\newcommand{\co}[1]{c({#1})\sb{\bullet}}
\newcommand{\cod}[1]{c\sb{+}({#1})\sb{\bullet}}
\newcommand{\cu}[1]{c({#1})\sp{\bullet}}
\newcommand{\szero}[1]{\sigma\sp{\ast}(#1)}
\newcommand{\Cd}[1]{\bC[{#1}]\sb{\bullet}}
\newcommand{\Fd}{\mathcal{F}\sb{\bullet}}
\newcommand{\Gd}{G\sb{\bullet}}
\newcommand{\Gu}{G\sp{\bullet}}
\newcommand{\oT}[1]{\overline{T}\sb{#1}}
\newcommand{\Ud}{\bU\sb{\bullet}}
\newcommand{\Vd}{\bV\sb{\bullet}}
\newcommand{\oV}[1]{\overline{V}\sb{#1}}
\newcommand{\Wd}{\bW\sb{\bullet}}
\newcommand{\Wdq}{\bW\sb{\bullet}\sp{\QQ}}
\newcommand{\oW}[1]{\overline{\bW}\sb{#1}}
\newcommand{\Wn}[2]{\bW\sb{#1}\bup{#2}}
\newcommand{\oWn}[2]{\overline{\bW}\sb{#1}\bup{#2}}
\newcommand{\W}[1]{\Wn{\bullet}{#1}}
\newcommand{\Wnk}[3]{\bW\sb{#1}\bup{#2,#3}}
\newcommand{\Wk}[2]{\Wnk{\bullet}{#1}{#2}}
\newcommand{\tW}{\widetilde{\bW}}
\newcommand{\tWn}[2]{\tW\quad\hspace*{-4mm}\sb{#1}\bup{#2}}
\newcommand{\tWd}[1]{\tWn{\bullet}{#1}}
\newcommand{\hW}{\widehat{\bW}}
\newcommand{\hWd}{\hW\sb{\bullet}}
\newcommand{\Xd}{\bX\sb{\bullet}}
\newcommand{\xd}{x\sb{\bullet}}
\newcommand{\hXd}{\hX\sb{\bullet}}
\newcommand{\Zd}{\bZ\sb{\bullet}}
\newcommand{\hZd}{\widehat{\bZ}\sb{\bullet}}
\newcommand{\bbj}{[\mathbf{j}]}
\newcommand{\bk}{[\mathbf{k}]}
\newcommand{\bkm}{[\mathbf{k}-\mathbf{1}]}
\newcommand{\bmm}{[\mathbf{m}]}
\newcommand{\bn}{[\mathbf{n}]}
\newcommand{\bnp}{[\mathbf{n}+\mathbf{1}]}
\newcommand{\bone}{[\mathbf{1}]}
\newcommand{\od}{\overline{\partial}}
\newcommand{\odz}[1]{\od\sb{#1}}
%
%
\newcommand{\cH}[3]{H\sp{#1}({#2};{#3})}
\newcommand{\HR}[2]{\cH{#1}{#2}{R}}
\newcommand{\HsR}[1]{\HR{\ast}{#1}}
\newcommand{\HF}[2]{\cH{#1}{#2}{\Fp}}
\newcommand{\HsF}[1]{\HF{\ast}{#1}}
%
%
\newcommand{\ma}[1][ ]{mapping algebra{#1}}
\newcommand{\Ama}[1][ ]{$\A$-mapping algebra{#1}}
\newcommand{\ABma}[1][ ]{$\A$-$\B$-bimapping algebra{#1}}
\newcommand{\Bma}[1][ ]{$\B$-dual mapping algebra{#1}}
\newcommand{\Tma}[1][ ]{$\bT$-mapping algebra{#1}}
\newcommand{\lin}[1]{\{{#1}\}}
\newcommand{\fff}{\mathfrak{f}}
\newcommand{\fM}{\mathfrak{M}}
\newcommand{\fMA}{\fM\sp{\A}}
\newcommand{\fMB}{\fM\sb{\B}}
\newcommand{\fMAB}{\fM\sp{\A}\sb{\B}}
\newcommand{\fMT}{\fM\sb{\bT}}
\newcommand{\fMR}{\fM\sb{R}}
\newcommand{\fT}{\mathfrak{T}}
\newcommand{\fTd}{\fT\sb{\bullet}}
\newcommand{\fV}{\mathfrak{V}}
\newcommand{\fVd}{\fV\sb{\bullet}}
\newcommand{\fVu}{\fV\sp{\bullet}}
\newcommand{\fW}{\mathfrak{W}}
\newcommand{\fWd}{\fW\sb{\bullet}}
\newcommand{\fX}{\mathfrak{X}}
\newcommand{\fY}{\mathfrak{Y}}
\newcommand{\fZ}{\mathfrak{Z}}
%
%
\newcommand{\PP}[2]{\mathcal{P}\sb{#1}[{#2}]}
\newcommand{\Pd}[1]{\PP{#1}{\delta}}
\newcommand{\Pg}[1]{\PP{#1}{\gamma}}
\newcommand{\Pe}[1]{P\!e\sb{#1}}
%
%
\newcommand{\uM}{\underline{M}}
\newcommand{\uN}{\underline{N}}
\newcommand{\uP}{\underline{P}}
%


\makeatletter
\@namedef{subjclassname@2020}{%
  \textup{2020} Mathematics Subject Classification}
\makeatother

%
\title{The higher structure of unstable homotopy groups}
\author{Samik Basu, David Blanc, and Debasis Sen}
\address{Stat-Math Unit, Indian Statistical Institute, Kolkata 700108, India}
\email{samikbasu@isical.ac.in}
\address{Department of Mathematics\\ University of Haifa\\ 3498838 Haifa\\ Israel}
\email{blanc@math.haifa.ac.il}
\address{Department of Mathematics \& Statistics\\
Indian Institute of Technology, Kanpur\\ Uttar Pradesh 208016\\ India}
\email{debasis@iitk.ac.in}

\date{\today}


\begin{abstract}
We construct certain unstable higher order homotopy operations indexed by the simplex
categories of \w{\Deln{n}} for \w[,]{n\geq 2} and prove that all elements in the
homotopy groups of a wedge of spheres are generated under such operations by Whitehead
products and the group structure. This provides a stronger unstable analogue of Cohen's
theorem on the decomposition of stable homotopy.
\end{abstract}

\maketitle

\setcounter{section}{-1}

%
%
\sect{Introduction}
\label{cint}

In \cite{JCohDS}, Joel Cohen showed how all elements in the stable homotopy groups
of the sphere spectrum can be generated under composition and higher order homotopy
operations from certain atomic elements: the three Hopf classes and their odd-primary
analogues. The main importance of this result is conceptual, as one of the few known
global facts about the stable homotopy groups as a whole
(see \cite{TLinHD,NisN,SchweSU,SchweSR,BSStapC,BSStapM}).
Our goal here is to show that a similar but stronger result holds for the unstable
homotopy groups of wedges of spheres.

There are a number of competing definitions of higher homotopy operations, particularly
in the unstable setting (see \cite{SpanS,BMarkH,ShipA,SagaU,CFranH,BJTurnHA}),
all of which ultimately involve obstructions to lifting appropriate diagrams from the
homotopy category. For our purposes, the relevant indexing categories turn out to be
finite subcategories of $\Delta$ (see \S \ref{snac}), yielding the notion of a
\emph{higher order simplicial operation}, defined in Section \ref{cshho}.
This construction makes sense in any \wwb{\infty,1}category, is in fact independent of
the specific model of $\infty$-categories we use, and includes as a special case
the long Toda brackets of \cite{GWalkL,BJTurnHI,BBSenT}. Formally, it also includes
the (iterated) suspension as a special case (see \S \ref{ssuspss} below).

We are mainly concerned here with decomposability for the homotopy groups of
spheres, since the decomposition for wedges of spheres follows from
Hilton's Theorem.  This can be thought of as a higher order ``ringoid'' version of
identifying the indecomposables for a graded algebra.  However,
following \cite[Theorem 4.5]{JCohDS}, we first address the ``module'' version: that is,
the question of higher order decomposability for \w{\pi\sb{\ast}\bX} as a \Pa
(a graded group equipped with an action of the primary homotopy operations \wh
see \cite[\S 1]{StoV}), in the following sense:

\begin{defn}\label{dgho}
Denote by \w{\Pi\sb{>1}} the graded set of all classes in \w[,]{\pis\bV}
for all finite wedges $\bV$ of simply-connected spheres.
For a simply-connected space \w[,]{\bX\in\Topa} we say that \w{\pi\sb{\ast}\bX} is
\emph{generated of higher order over} \w{\Pi\sb{>1}} by a graded subset
\w{A\subset\pi\sb{\ast}\bX} if any \w{\gamma\in\pi\sb{N}\bX} is a linear combination
over $\ZZ$  of elements obtained from $A$ recursively
\begin{enumerate}
\renewcommand{\labelenumi}{(\roman{enumi})}
\item By the action of a non-trivial primary operation \wh that is,  a Whitehead product
or precomposition with an element \w{\alpha\in\pi\sb{\geq N+1}\bS{N}} for some \w{N\geq 2}
(see \cite[Ch.\ XI]{GWhE});  or
\item As the value of the higher order simplicial operation associated to an augmented
restricted simplicial space \w{\Wd\to\bX} as in Section \ref{cssa}, with each
\w{\bW\sb{n}} a wedge of simply-connected spheres.
\end{enumerate}
\noindent Note that both types of constructions raise degrees.
\end{defn}

\begin{defn}\label{datomicc}
The set $A$ of \emph{atomic} classes in \w{\pi\sb{\ast}\bX}
consists of those \w{[f]\in\pi\sb{n}\bX} which, for some prime \w[,]{p\geq 0}
induce a non-zero map \w{f\sb{\ast}:\HiF{n}{\bS{n}}\to\HiF{n}{\bX}}
(where \w{\FF\sb{0}:=\QQ} by convention). These classes are determined up to
multiplication by \w{1<k<p} when \w[.]{p>0}
\end{defn}

\begin{thma}
If  $\bX$ is simply-connected, the elements of \w{\pis\bX} are generated of
higher order over \w{\Pi\sb{>1}} by the atomic maps.
\end{thma}
\noindent See Theorem \ref{tjct} below. Our main technical result is then:

\begin{thmb}
If $\bX$ is the $k$-connected cover of the $k$-sphere for \w[,]{k\geq
2} the atomic maps in \w{\pis\bX} consist of (the lifts of) the Hopf maps
\w[,]{\eta\sb{k}\in\pi\sb{k+1}\bS{k}}
\w[,]{\nu\sb{k}\in\pi\sb{k+3}\bS{k}} and
\w[,]{\sigma\sb{k}\in\pi\sb{k+7}\bS{k}} and the first $p$-torsion
classes \w[,]{\alpha\sb{1}(p)\in\pi\sb{k+2p-3}\bS{k}} for odd primes $p$.
\end{thmb}
\noindent See Theorem \ref{ttransg} below. The ``spheres of birth'' for the various
classes differ, of course: \w{k=2} for \w[,]{\eta\sb{k}} \w{k=4} for \w[,]{\nu\sb{k}}
\w{k=8} for \w[,]{\sigma\sb{k}} and \w{k=3} for \w[.]{\alpha\sb{1}(p)}

In this paper we are mainly concerned with the ``ringoid'' \w{\Pi\sb{>1}} itself,
whose elements correspond to the primary operations on homotopy groups (as noted above).

\begin{defn}\label{dgoho}
We say that \w{\Pi\sb{>1}} is \emph{generated of higher order} by a (graded)
subset \w{\Fc{0}} if there is an exhaustive increasing filtration
\w{\Fc{0}\subseteq\Fc{1}\subseteq\Fc{2}\subseteq\dotsc} of
\w[,]{\Pi\sb{>1}} such that \w{\Fc{n}} is obtained from
classes in lower filtration by group operations, composition (of elements in
\w[),]{\Fc{<n}} or as the value of a higher order simplicial operation associated to a
simplicial space \w{\Wd} with \w[,]{\|\Wd\|\simeq\bV} where that
each \w{\bW\sb{n}} is a finite wedge of spheres and each face map of \w{\Wd}
is in \w[.]{\Fc{<n}}
\end{defn}

Theorem B then implies that \w{\Pi\sb{>1}} is generated of higher order by the atomic
maps listed there, yielding an unstable version of \cite[Theorem 4.2]{JCohDS} (see
Corollary \ref{catomic}).  A further analysis yields:

\begin{thmc}
The homotopy groups of a finite wedge of simply-connected spheres are generated
of higher order by the fundamental classes (i.e., inclusions of wedge summands)
and their Whitehead products.
\end{thmc}
\noindent See Theorem \ref{twhitpgen} below\vsm.

In addition to the sample calculations of Section \ref{csusp}, used in the proof of
Theorem C, in Section \ref{ccpn} we show what form it takes in practice
by providing a rational higher operation description of the Hopf fibrations
\w{\bS{2n+1}\to\CP{n}} for all complex projective spaces.

\begin{remark}\label{ralg}
Cohen actually provides an algorithm for finding decompositions of stable
classes, using spectral sequences (and conversely, see \cite{AdHI,WXuT} for examples
of the use of such decompositions to calculate differentials in the
Adams spectral sequence).

It seems unlikely that such an algorithm could be obtained unstably with the
present state of our knowledge. However, in \cite{BBSenA} we provide a more detailed
description of how the value of a previously defined higher order operation may be
inserted in the diagram defining another such operation, with a view to defining
an explicit notion of the ``algebra of unstable higher homotopy operations''. This
is perhaps more meaningful than Theorem C itself, since the abstract fact that a
particular set \w{\Fc{0}} suffices to generate \w{\Pi\sb{>1}} is not very useful on
its own.
\end{remark}

\begin{notn}\label{snac}
Let $\Delta$ denote the category of non-empty finite ordered sets and order-preserving
maps (cf.\ \cite[\S 2]{MayS}), and \w{\Dres} the subcategory with the same
objects but only monic maps. Similarly, \w{\Dp} denotes the category of all
finite ordered sets (including $\emptyset$), \w{\Dresp} the corresponding subcategory
of monic maps, and \w{\Drn{n}} the subcategory of \w{\Dresp} with consisting of
ordered sets with at most \w{n+1} elements.

A \emph{simplicial object} \w{\Gd} in a category $\C$ is a functor
\w[,]{\Dop\to\C} a \emph{restricted} simplicial (or semi-simplicial) object is a
functor \w[,]{\Dres\op\to\C} an \emph{augmented} simplicial object is a functor
\w[,]{\Dp\op\to\C} a \emph{restricted augmented} simplicial object is a functor
\w[,]{\Dresp\op\to\C} and an
\emph{$n$-truncated restricted augmented} simplicial object is a functor
\w[.]{\Drnop{n}\to\C}

In each case  write \w{G\sb{n}} for the value of \w{\Gd} at
\w[.]{\bn=(0<1<\dotsc<n)} There is a natural embedding
\w[,]{\co{-}:\C\to\C\sp{\Dop}} with \w{\co{A}} the constant
simplicial object and similarly \w{\cod{A}} for the constant augmented
simplicial object. The inclusion of categories \w{\Delta\to\Dp} induces
the forgetful functor \w[.]{\C\sp{\Dp\op} \to \C\sp{\Dop}}

The category of topological spaces will be denoted by \w[,]{\Top}
that of pointed spaces by \w[,]{\Topa} and that of pointed connected
spaces by \w[.]{\Tz} The category of simplicial sets will be
denoted by \w[,]{\Ss=\Set\sp{\Dop}} that of pointed simplicial
sets by \w[,]{\Sa=\Seta\sp{\Dop}} that of reduced simplicial
sets by \w{\Sr} (see \cite[III, \S 3]{GJarS}), and that of simplicial groups
by \w[.]{\G=\Grp\sp{\Dop}} Recall that a reduced simplicial set \w{\bX} is one
that has a unique zero simplex \w[.]{X_0 = \ast} Write \w{\mapa(\bX,\bY)} for
the standard function complex in \w[,]{\Sa} \w[,]{\Tz} or
$\G$ (see \cite[I, \S 1.5]{GJarS}).
\end{notn}

\begin{ack}
  We would like to thank the referee for his or her carefully formulated comments,
  and George Peschke for pointing out that the rational case was missing in Definition
  \ref{datomicc}.
The first author was partially supported by SERB MATRICS grant 2018/000845,
and the second author by Israel Science Foundation grant 770/16.
\end{ack}

%
%
\sect{Simplicial higher order operations}
\label{cshho}

As noted in the introduction, many versions of higher order operations have
appeared in the literature, starting with Adem's secondary cohomology operations
(see \cite{AdemI}), Massey's triple products (see \cite{MassN}), and the Toda
brackets of \cite{TodG}.
The values of such operations usually appear as obstructions to lifting a
homotopy-commutative diagram to a model category $\C$,
and are characterized by the following properties:
\begin{enumerate}
\renewcommand{\labelenumi}{(\roman{enumi})}
\item The operations are \emph{indeterminate}, in the sense that more than
one value may be obtained, depending on choices made in lifting;
\item An $n$-th order operation is defined only when all lower order operations
(associated to partial diagrams) vanish, for some consistent set of choices.
\item The final values are expressible in \w[.]{\ho\C}
\end{enumerate}
\noindent See \cite{BMarkH,BJTurnHA}.

\begin{mysubsection}{Models of $\infty$-categories}
\label{smic}
Although in this paper we are only concerned with topological spaces, the notion of
higher operation we use here makes sense in a more general setting \wh
namely, any model of $\infty$-category theory consisting of:

\begin{enumerate}
\renewcommand{\labelenumi}{(\alph{enumi})}
\item A category $\C$, with a distinguished full subcategory \w{\C\sb{0}} of
$\infty$-\emph{categories}.
\item A \emph{homotopy category} functor \w[,]{\Pc\colon\C\sb{0}\to\Cat}
with a right adjoint \w{B:\Cat\to\C\sb{0}} called the \emph{nerve} functor.
\item A \emph{set of objects} functor \w[,]{\Ob\colon\C\to\Set} such that for
each \w{X\in\C\sb{0}} and \w[,]{x, y \in \Ob(X)} we have a Kan complex
\w{\Map\sb{X}(x, y)} with homotopy associative and unital composition, and
\w[.]{\Pc(x, y)=\pi\sb{0}\Map\sb{X}(x, y)} Morphisms
\w{F\colon X\to Y} in $\C$ induce
\w[,]{\Map\sb{X}(x, y) \rightarrow \Map\sb{Y}(F(x), F(y))}
which respects the  composition operation up to homotopy.
\item When dealing with any specific \w[,]{X\in\C\sb{0}} we shall assume that all
necessary limits and colimits exist in $X$, and that it is \emph{pointed} (that is,
the initial and final objects coincide).
\end{enumerate}
\noindent Compare \cite{RVeriIC} and \cite[\S 2]{BMeadS}.

In To{\"{e}}n's axiomatization (see \cite[\S 4]{ToenA}), and in all examples of interest,
$\C$ is a model category and \w{\C\sb{0}} consists of the fibrant objects in $\C$.
For example, if \w{\C=\Ss} with the Joyal model category, \w{\C\sb{0}} consists of
the quasi-categories, while if $\C$ is simplicial categories, \w{\C\sb{0}} consists of
those enriched in Kan complexes.
One then has a Quillen equivalence of $\C$ to the complete Segal model structure of
\cite{RezkM}, and therefore to quasi-categories (\cite{LuriH}),
simplicial categories (\cite{BergM}), and the other standard models of
\wwb{\infty,1}categories.
\end{mysubsection}

\begin{mysubsection}{The homotopy spectral sequence of a simplicial object}
\label{shssso}
Our definition of higher operations is inspired by the description of
the differentials in the homotopy spectral sequence of a simplicial space in
\cite[\S 6]{BMeadS}, which we briefly recall:

For \w{(\C,\C\sb{0})} and \w{X\in\C\sb{0}} as in \S \ref{smic}, let \w{\xd}
be a simplicial object in $X$ (if \w[,]{\C=\SC} then \w{\xd} is just
an $\infty$-homotopy commutative diagram in the sense of \cite[\S 2.3]{DKSmH}).
If $y$ is a homotopy cogroup object in $X$, then \w{\hWd:=\Map\sb{X}(y,\xd)} is an
$\infty$-homotopy commutative simplicial object in $\Ss$ which may be made into
a strict simplicial object \w{\Wd\in\Ss\sp{\Dop}} by \cite[Corollary 2.5]{DKSmH}
(or \cite[Theorem 4.49]{BVoHI}).
The \emph{homotopy spectral sequence} of \w{(\xd,y)} is then defined to be
the Bousfield-Friedlander spectral sequence for \w{\Wd} (see \cite[Theorem B.5]{BFrieH}),
having
\begin{myeq}[\label{eqssss}]
\Ett{n}{k}~\cong~\pi\sb{n}\sp{h}\pi\sb{k}\sp{v}\Wd~\Longrightarrow~\pi\sb{n+k}\|\Wd\|~,
\end{myeq}
\noindent where \w[,]{\|\Wd\|} the diagonal of the bisimplicial set, is also its
homotopy colimit (see \cite[XII, \S 2]{BKanH}), so it is weakly equivalent to
\w{\Map\sb{X}(y,\colim\sb{\Dop}\xd)} (assuming $X$ has enough colimits).

The original version, for bisimplicial groups, is due to Quillen
(see \cite{QuiS}). If we apply geometric realization to \w{\Wd} in each simplicial
dimension, we obtain a simplicial topological space \w[,]{\hXd} with Reedy fibrant
replacement \w{\Xd} (see \cite[\S 15.3]{PHirM}). Dwyer, Kan, and Stover constructed
the \emph{spiral spectral sequence} of \w{\Xd} (with each \w{\bX\sb{n}} connected),
and showed in \cite[Proposition 8.3]{DKStB} that it is isomorphic to the above from
the \ww{E\sp{2}}-term on.

In \cite{BMeadS} we showed that this spectral sequence can be set up internally
to the $\infty$-category $X$, in terms of \w{(\xd,y)} themselves: for this purpose,
we represent any element in the \ww{E\sp{2}}-term of \wref{eqssss} by a map
\w[,]{f:\Sigma\sp{k}y\to x\sb{n}} and show that it survives to the \ww{E\sp{r}}-term
if and only if we can include $f$ in a (homotopy coherent) diagram of the form
\myudiag[\label{eqerrep}]{
\Sigma\sp{k}y  \ar@/^1pc/[rr]^{0}_{\vdots}  \ar@/_1pc/[rr]_{0} \ar[dd]_{f}  &&
0  \ar@/^1pc/[rr]_{\vdots}  \ar@/_1pc/[rr] \ar[dd] && 0 \ar[dd] & \cdots \cdots & 0 \ar[dd]\\
&& && & & \\
x\sb{n} \ar@/^1pc/[rr]^{d\sb{0}}\sb{\vdots}  \ar@/_1pc/[rr]\sb{d\sb{n}} && x\sb{n-1}
\ar@/^1pc/[rr]^{d\sb{0}}\sb{\vdots}  \ar@/_1pc/[rr]\sb{d\sb{n-1}} && x\sb{n-2} &
\cdots \cdots & x\sb{n-r+1}
}
\noindent in $X$.

Moreover, as shown in \cite[Theorem 6.8]{BMeadS}, the value of the differential
\w{d\sp{r}\lra{f}\in\Eu{r}{n-r}{k+r-1}} is represented in \w{\Eot{n-r}{k+r-1}} by
homotopy classes of various maps \w[,]{\alpha:\Sigma\sp{k+r-1}y\to x\sb{n-r}}
each obtained as a value of a certain $r$-th order homotopy operation
associated to \wref[.]{eqerrep}

In particular, if \w{[f]\in\Eot{n}{k}} is a permanent cycle, the element it represents
in \w{\pi\sb{n+k}\|\Wd\|\cong[\Sigma\sp{n+1}y,\colim\xd]\sb{X}} may be described in
precisely the same way as the value of an \wwb{n+1}st order operation.
\end{mysubsection}

\begin{mysubsection}{Differentials in the spectral sequence}
\label{sdss}
It is simplest to describe the differentials in \w[,]{\C=\SC} the category of simplicially
enriched categories (or equivalently, simplicial categories with constant object sets),
with \w{\C\sb{0}} thus consisting of categories enriched in Kan complexes.
However, it is important to point out that the construction makes sense in any model
of \wwb{\infty,1}categories satisfying the assumptions of
\S \ref{smic}. This follows from \cite[Corollary 6.11]{BMeadS}, and is illustrated
for quasi-categories in \cite[\S 8]{BMeadS}.

Thus we work directly with a homotopy-coherent simplicial space \w[.]{\Zd}
This means replacing the indexing category $\Delta$ by a cofibrant replacement
in \w[,]{\SC} such as the Dwyer-Kan resolution \w{\DK(\Delta)} (see \cite[\S 2]{DKanS}).
In fact, we can replace $\Delta$ by \w{\Dres} (see \S \ref{snac}), since our
spectral sequence  is actually determined by the restriction \w{\Zd'} of \w{\Zd}
to \w[.]{\Drop}

Now if \w{\hZd} is any strictification of \w[,]{\Zd} they have the same (homotopy)
colimit in the \wwb{\infty,1}category of spaces, so by abuse of notation we may
denote this colimit by \w{\|\Zd\|} (since, as noted above, it may be identified with the
diagonal \w[).]{\|\hZd\|} Moreover, \w{\|\Zd\|\simeq\colim\sb{\Drop}\Zd'}
by \cite[Appendix A]{SegCC}.

Recall that the $n$-\emph{permutohedron} \w{P\sp{n}} is the convex hull of the
\w{(n+1)!} points in \w{\RR\sp{n+1}} obtained by permuting the coordinates of
\w{(x\sb{0},\cdots,x\sb{n})} for \w{n+1} distinct real numbers
\w[.]{\{x\sb{0},\cdots,x\sb{n}\}}
Thus the $1$-permutohedron is a $1$-simplex, and the $2$-permutohedron
is a hexagon:
\myfigure[\label{eqhexagon}]{
\begin{picture}(100,110)(-100,-10)
%
%
\put(95,85){\line(-3,-2){30}}
\put(37,79){\scriptsize $(0,1)\times(2)$}
\put(97,87){\circle*{3}}
\put(80,90){\small $(1,0,2)$}
%
%
\put(100,85){\line(3,-2){30}}
\put(117,79){\scriptsize $(1)\times(0,2)$}
\put(132,63){\circle*{3}}
\put(135,60){\small $(1,2,0)$}
%
%
\put(132,60){\line(0,-1){30}}
\put(134,42){\scriptsize $(1,2)\times(0)$}
\put(132,28){\circle*{3}}
\put(135,20){\small $(2,1,0)$}
%
%
\put(130,26){\line(-3,-2){30}}
\put(113,8){\scriptsize $(2)\times(0,1)$}
\put(97,5){\circle*{3}}
\put(80,-6){\small $(2,0,1)$}
%
%
\put(64,28){\circle*{3}}
\put(26,25){\small $(0,2,1)$}
\put(65,26){\line(3,-2){30}}
\put(36,8){\scriptsize $(0,2)\times(1)$}
%
%
\put(64,63){\circle*{3}}
\put(26,60){\small $(0,1,2)$}
\put(64,60){\line(0,-1){30}}
\put(19,42){\scriptsize $(0)\times(1,2)$}
\end{picture}
}

By \cite[Proposition 5.6]{BMeadS}, for every
\w{-1\leq j<m} there is an isomorphism of simplicial sets
\begin{myeq}[\label{eqpermut}]
\DK(\Dresp\op)(\bmm,\bbj)~\cong~\coprod\sb{\theta:\bbj\to\bmm} P\sp{j-m-1}
\end{myeq}
\noindent between the mapping space in \w{\DK(\Dresp)} and a disjoint
union of (triangulations of) the \wwb{j-m-1}-permutohedron, indexed by the distinct maps
\w{\bbj\to\bmm} in \w[.]{\Dres}

This allows us to reduce the search for a diagram of the form \wref{eqerrep} in our
$\infty$-category $X$ to the case where \w[,]{X=\SC} \w{\xd} is replaced by the
(strict) restricted simplicial space \w{\Wd} realizing
\w[,]{\map\sb{X}(y,\xd')} and \w{\Sigma\sp{i}y} in $X$ is replaced by a
sphere \w[.]{\bS{k}} However, the source (upper horizontal diagram in \wref[)]{eqerrep}
is still required to be cofibrant in \w[,]{\SC} and all the necessary higher
homotopies needed to provide the coherence of \wref{eqerrep} may be encoded
by adjunction in a pointed map from the $r$-skeleton of \wref{eqpermut} into
\w[.]{\map\sb{\Sa}(\bS{k},\bW\sb{n-r+1})}

Moreover, by a more careful analysis of the spiral spectral sequence in the case where
\w{\Wd} is Reedy fibrant, one can show that only one of the components in the right hand
side of \wref{eqpermut} is needed (see \cite[Theorem 6.8]{BMeadS}). Using the fact that
\w{P\sp{n}} is a convex polytope in \w[,]{\RR\sp{n}} so its boundary
(or \wwb{n-1}skeleton) is an \wwb{n-1}sphere, by a further adjunction we obtain a
single map \w[,]{\alpha:\bS{i+r}\to\bW\sb{n-r+1}} which represents \w{d\sp{r}\lra{f}} in
\w{\Eu{r}{n-r+1}{i+r}} by \cite[Corollary 6.10]{BMeadS}.

As explained in \cite[Corollary 6.11]{BMeadS}, the construction sketched above may be
described in terms of a sequence of maps in the original $\infty$-category $X$,
starting with the diagram \wref{eqerrep} in the homotopy category of $X$, with each map
determined by the universal property of a colimit on an appropriate subcategory of
\w{\Dres} in terms of the maps obtained inductively in earlier stages. For our purposes
we shall not need the details of this construction, but only the following:
\end{mysubsection}

\begin{defn}\label{dshoho}
Let \w{X\in\C\sb{0}} be an $\infty$-category as in \S \ref{smic},
\w{\xd:B\Dresp\op\to X} an augmented restricted simplicial object in $X$, and $y$
a cogroup object in \w[.]{\ho X} For each \w[,]{n\geq 1} the associated
\emph{simplicial \wwb{n+1}st order operation} is defined as follows

\begin{enumerate}
\renewcommand{\labelenumi}{(\alph{enumi})~}
\item The \emph{initial data} for the operation consists of \w[,]{\xd} together
with the homotopy class of a map \w{f:\Sigma\sp{k}y\to x\sb{n}} in $X$
with \w{d\sb{j}f\sim 0} for \w[.]{0\leq j\leq n}
\item \emph{Full data} for the operation consists of a diagram in $X$:
\myudiag[\label{eqfulldata}]{
\Sigma\sp{k}y  \ar@/^1pc/[rr]^{0}_{\vdots}  \ar@/_1pc/[rr]_{0} \ar[dd]_{f}  &&
0  \ar@/^1pc/[rr]_{\vdots}  \ar@/_1pc/[rr] \ar[dd] && 0 \ar[dd] & \cdots \cdots &
0 \ar[dd] &&\\
&& && & & && \\
x\sb{n} \ar@/^1pc/[rr]^{d\sb{0}}\sb{\vdots}  \ar@/_1pc/[rr]\sb{d\sb{n}} && x\sb{n-1}
\ar@/^1pc/[rr]^{d\sb{0}}\sb{\vdots}  \ar@/_1pc/[rr]\sb{d\sb{n-1}} && x\sb{n-2} &
\cdots \cdots & x\sb{0} \ar[rr]\sp{\vare} && x\sb{-1}
}
\noindent (implicitly involving choices for all higher coherences), if it exists.
\item The \emph{value} for the operation is the homotopy class of the map
\w{\Sigma\sp{n+k}y\to x\sb{-1}} determined by the full data, and the fact that
the (homotopy) colimit of the top row of \wref{eqfulldata} is \w[.]{\Sigma\sp{n+k}y}
\end{enumerate}

It is evident from the usual properties of spectral sequences that the
simplicial higher order homotopy operations we have defined satisfy the three properties
listed at the beginning of Section \ref{cshho}.
Moreover, when $X$ is an $\infty$-category model for \w[,]{\Topa} $y$ is a sphere, and
each \w{x\sb{n}} is a wedge of spheres, we have in fact a higher homotopy operation
\textit{sensu stricto}, in as much as all the ingredients of the construction take
values in homotopy groups.
\end{defn}

\begin{remark}\label{rextaso}
Note that a diagram of the form \wref{eqfulldata} (or \wref[)]{eqerrep}
in $X$ or \w{\ho X} can be re-written as a single truncated restricted simplicial object
\myudiag[\label{eqextaso}]{
\Sigma\sp{k}y \ar@/^1pc/[rr]\sp{d\sb{0}=f}\sb{\vdots} \ar@/_1pc/[rr]\sb{d\sb{n+1}=0} &&
x\sb{n} \ar@/^1pc/[rr]^{d\sb{0}}\sb{\vdots}  \ar@/_1pc/[rr]\sb{d\sb{n}} && x\sb{n-1}
\ar@/^1pc/[rr]^{d\sb{0}}\sb{\vdots}  \ar@/_1pc/[rr]\sb{d\sb{n-1}} && x\sb{n-2} &
\cdots \cdots & x\sb{-1}
}
\noindent extended one further degree to the left, with all face maps
out of the \wwb{n+1}slot (except \w[)]{d\sb{0}} equal to $0$.

This was the original point of view  in the construction of the spiral spectral
sequence of a simplicial space \w[:]{\Wd\in\Topa\sp{\Dop}} Dwyer, Kan, and Stover
showed that if \w{\Wd} is Reedy fibrant, there is a fibration sequence
\begin{myeq}[\label{eqfibseq}]
  \Omega Z\sb{n-1}\Wd~\xra{\partial\sb{n-1}}~Z\sb{n}\Wd~\xra{j\sb{n}}~
  C\sb{n}\Wd~\xra{\bd\sp{W\sb{n}}}~Z\sb{n-1}\Wd~.
\end{myeq}
\noindent for each \w[,]{n\geq 1} where
\w{C\sb{n}\Wd:=\bigcap\sb{i=1}\sp{n}\ \Ker(d\sb{i})} is the $n$-th Moore chains
space, and
\w{Z\sb{n}\Wd:=\bigcap\sb{i=0}\sp{n}\ \Ker(d\sb{i})} is the $n$-th Moore cycles space.

The spiral spectral sequence for \w{\Wd} is that associated to the tower of
fibrations \wref{eqfibseq} (see \cite{DKStB} for further details).
It is then clear that \w{[f]\in \pi\sb{k}C\sb{n}\Wd}
survives to \w{E\sp{\infty}\sb{n,k}} if and only if it lifts to \w[,]{Z\sb{n}\Wd}
thus strictifying \wref{eqextaso} (or \wref[).]{eqfulldata} We deduce:
\end{remark}

\begin{prop}\label{pfiltrationn}
An element \w{\gamma\in \pi\sb{N}\|\Wd\|} is in filtration $n$ of the
homotopy spectral sequence for \w{\Wd} if and only if it is a canonical value for
the \wwb{n+1}st order simplicial operation associated to \wref[,]{eqfulldata}
for \w[.]{N=n+k}
\end{prop}

\begin{mysubsection}{Toda brackets}
\label{stb}
We have used the homotopy spectral sequence of a restricted simplicial object as
a convenient shorthand for defining our higher order operations. However, they have
an explicit geometric-combinatorial construction which is independent of \wh and more
general than \wh such spectral sequences (see \cite{BJTurnHA} and the references
there for more details). We illustrate this for the oldest example of a secondary
operation: the ordinary Toda bracket of \cite{TodG,TodC}.

Assume that we are given a sequence of maps
\begin{myeq}\label{eqtoda}
\bX~\xra{h}~\bY~\xra{g}~\bZ~\xra{f}~\bW
\end{myeq}
\noindent in a pointed model category $\C$, with \w{f\circ g\sim\ast} and
\w[.]{g\circ h\sim\ast}
This can be rewritten as a $2$-truncated restricted augmented simplicial object
in \w[:]{\ho\C}
\mydiagram[\label{eqsimptodah}]{
  \bX \ar@/^{1.2pc}/[rr]\sp{d\sb{0}=h} \ar[rr]\sp{d\sb{1}=\ast}
  \ar@/_{0.7pc}/[rr]\sb{d\sb{0}=\ast} &&
  \bY \ar@/^{0.5pc}/[rr]\sp{d\sb{0}=g} \ar@/_{0.5pc}/[rr]\sb{d\sb{1}=\ast} &&
  \bZ\ar[rr]\sp{\vare=f} &&   \bW~,
}
\noindent as in \wref[.]{eqextaso}
We can try to make \wref{eqsimptodah} into a strict restricted simplicial
object as follows:
\mydiagram[\label{eqsimptodaexp}]{
\bX \ar[rr]\sp{d\sb{0}=h} \ar@{_{(}->}[rrd]\sp{d\sb{1}=\inc} \ar@{_{(}->}[rrdd]\sb{d\sb{2}=\inc} &&
\bY \ar[rr]\sp{d\sb{0}=g} \ar@{_{(}->}[rrd]\sb(0.2){d\sb{1}=\inc} &&
\bZ\ar[rrd]\sp{\vare=f} && \\
&& C\bX \ar[rru]\sb(0.8){d\sb{0}=G} \ar@{_{(}->}[rrd]\sp(0.3){d\sb{1}=\inc\sb{1}} &&
C\bY \ar[rr]\sp{\vare=F} &&\bW\\
&& C\bX \ar[rru]\sb(0.8){d\sb{0}=Ch} \ar@{_{(}->}[rr]\sb{d\sb{1}=\inc\sb{2}} &&
\Sigma\bX \ar[rru]\sb{\vare=\lra{f,g,h}}
}
\noindent where the space in each simplicial dimension, from left to right, is the wedge
of the relevant column, and \w{\inc\sb{i}:C\bX\hra\Sigma\bX} \wb{i=1,2} denote the two
inclusions of the upper and lower cones into \w{\Sigma\bX} (the pushout of
\w[).]{C\bX\leftarrow\bX\to C\bX}
Here \w{\lra{f,g,h}} is by definition the value of the Toda bracket associated to the two
choices of nullhomotopies \w{F:f\circ g\sim\ast} and \w{G:g\circ h\sim\ast}
(uniquely defined by the requirement that \wref{eqsimptodaexp} satisfies the
simplicial identities on the nose).

We see that \wref{eqsimptodaexp} realizes \wref{eqsimptodah} up to homotopy if
and only if the Toda bracket vanishes, so that we can choose a nullhomotopy
\w[,]{K:\lra{f,g,h}\sim\ast} allowing us to replace \w{\Sigma\bX} by \w[,]{C\Sigma\bX}
with \w{\vare=K} on the cone.

See \cite{BBSenT} for a detailed treatment of Toda brackets of arbitrary length
in a similar spirit (and compare \cite{BBGondH}).
\end{mysubsection}

%
%
\sect{Simplicial space approximations}
\label{cssa}

We would like to have a procedure for decomposing elements in the homotopy groups
\w{\pis\bX} of an arbitrary space $\bX$ in terms of higher order homotopy operations.
By the description in \S \ref{sdss}, we can do so by providing
a suitable simplicial space \w{\Wd} with an augmentation to $\bX$ \wh most simply,
if \w[.]{\|\Wd\|\simeq\bX}

\begin{mysubsection}{Constructing CW approximations}
\label{sccwa}
For simplicity we may assume $\bX$ is a \wwb{k-1}connected pointed space
with \w[.]{k\geq 2} We define a \emph{sequential approximation} to $\bX$
to be a sequence
\begin{myeq}\label{eqtower}
\W{k}~\xra{\prn{k}}~\W{k+1}~\xra{\prn{k+1}}~\W{k+2}~\to~\dotsc~
\W{n}~\xra{\prn{n}}~\W{n+1}~\to~\dotsc~
\end{myeq}
\noindent of simplicial spaces with augmentations
\w{\bve{n}:\W{n}\to\bX} (commuting with the maps \w[),]{\prn{n}} such that
  for each \w[:]{n\geq k}

\begin{enumerate}
\renewcommand{\labelenumi}{(\alph{enumi})~}
\item \w{\W{n}} is \wwb{n-k}skeletal (in the simplicial direction).
\item For each \w[,]{i\geq 0} \w{\Wn{i}{n}} is homotopy equivalent to a
  wedge of \wwb{k-1}connected spheres.
\item The map \w{\var{n}:\|\W{n}\|\to\bX} induced by \w{\bve{n}} (out of
  the geometric realization) is an $n$-equivalence: that is, it induces
  an isomorphism in \w{\pi\sb{i}} for \w{i<n} and an epimorphism for \w[).]{i=n}
  This implies that any $n$-skeleton of \w{\|\W{n}\|} is an $n$-skeletal
  CW approximation for $\bX$.
\end{enumerate}
\end{mysubsection}

Recall that a space is of \emph{finite type} if each homotopy group is finitely generated. For a simply-connected space, this condition is equivalent to each homology group being finitely generated.

\begin{prop}\label{pseqapp}
  Every \wwb{k-1}connected pointed space \wb{k\geq 2} has a sequential approximation; if
  $\bX$ is of finite type, we may assume each \w{\Wn{i}{n}} is homotopy equivalent
  to a finite wedge of simply-connected spheres.
\end{prop}

\begin{proof}
We construct \wref{eqtower} by induction on \w[\vsm:]{n\geq 0}

\noindent\textbf{Step 1:}\hs
For \w[,]{n=k} start with \w{\Wn{0}{k}} a wedge of $k$-spheres with a map
\w{\bve{k}:\Wn{0}{k}\to\bX} which induces a surjection in \w[.]{\pi\sb{k}}
We choose it to be minimal with this property (that is, no proper sub-wedge
has such a surjection). We let \w{\W{k}} be \w[\vsm.]{\co{\Wn{0}{k}}}

\noindent\textbf{Step 2:}\hs
For \w[,]{n=k+1} choose \w{\oWn{1}{k+1}} to be a (minimal) wedge of $k$-spheres having
a map \w{d\sb{0}:\oWn{1}{k+1}\to\Wn{0}{0}} inducing a surjection onto the kernel of
\w[.]{\bve{k}\sb{\#}:\pi\sb{k}\Wn{0}{k}\to\pi\sb{k}\bX} Let \w{F:C\Wn{1}{k+1}\to\bX}
be a nullhomotopy for \w[,]{\bve{k}\circ d\sb{0}} and set
\w{\tWn{0}{k+1}:=\Wn{0}{k}\vee C\Wn{1}{k+1}}
and \w[.]{\tWn{1}{k+1}:=\oWn{1}{k+1}\vee\tWn{0}{k+1}} We obtain a $1$-skeletal simplicial
space \w[,]{\tWd{k+1}} with an augmentation to $\bX$ given by \w[.]{\bve{k}\bot F}
The degeneracy \w{s\sb{0}:\tWn{0}{k+1}\to\tWn{1}{k+1}} is given by the inclusion.
The map \w{\tvar{k+1}:\|\tWd{k+1}\|\to\bX} thus induces an isomorphism in
\w[.]{\pi\sb{k}}

Now choose a (minimal) wedge of \wwb{k+1}spheres \w{\oWn{0}{k+1}} with a map
\w{e:\oWn{0}{k+1}\to\bX} mapping onto
\w[.]{\pi\sb{k+1}\bX\setminus\Image(\tvar{k+1}\sb{\#})}
Thus if we set \w{\Wn{0}{k+1}:=\tWn{0}{k+1}\vee\oWn{0}{k+1}} and
\w[,]{\Wn{1}{k+1}:=\tWn{1}{k+1}\vee\Wn{0}{k+1}} we obtain a $1$-skeletal
simplicial space \w{\W{k+1}} with an augmentation to $\bX$ such
that the induced map \w{\var{k+1}:\|\W{k+1}\|\to\bX} induces a surjection in
\w{\pi\sb{k+1}} (and so is a \wwb{k+1}equivalence)\vsm.

\noindent\textbf{Step 3:}\hs Assume given \w{\W{n-1}} as above, and let
\w{m\geq n-1} be maximal such that the map \w{\var{n-1}:\|\W{n-1}\|\to\bX}
(induced by \w[)]{\bve{n-1}:\W{n-1}\to\bX} is an $m$-equivalence.
If \w[,]{m\geq n} we set \w[;]{\W{n}:=\W{n-1}}
otherwise \w[,]{m=n-1} and we construct \w{\W{n}} by an inner induction as follows:

\begin{enumerate}
\renewcommand{\labelenumi}{Case \arabic{enumi}.~}
\item If \w{\pi\sb{n-1}\vare:\pi\sb{n-1}\|\W{n-1}\|\to\pi\sb{n-1}\bX} is injective
  (and thus an isomorphism), necessarily \w{\pi\sb{n}\vare} is not surjective. Choose
  a minimal wedge of $n$-spheres \w{\bW'} and a map \w{\vare':\bW'\to\bX}
  mapping onto \w[,]{\pi\sb{n}\bX\setminus\Image(\vare\sb{\#})} and set
  \w[,]{\Wn{0}{n}:=\Wn{0}{n-1}\vee\bW'} with the necessary degeneracies in higher
  simplicial dimensions. Note that \w{\vare'} has non-trivial Hurewicz image.
\item If \w{\pi\sb{n-1}\var{n-1}} is not injective,
  let \w[,]{K:=\Ker(\pi\sb{n-1}\var{n-1})} with \w{j\geq 0} the minimal filtration
  of elements of $K$, and set \w{\Wk{n}{i}:=\W{n-1}} for each \w[,]{i\leq j} and
  \w{\Wnk{i}{n}{j+1}:=\Wn{i}{n}} for \w{0\leq i<j}
  Next, choose a minimal set of generators \w{\{\alpha\sb{i}\}\sb{i\in I}} for
  the elements of $K$ in filtration $j$, represented by maps
  \w{g\sb{i}:\bS{n\sb{i}}\to Z\sb{j}\Wnk{j}{n}{0}} (see \S \ref{rextaso}), and let
  \w[.]{\Wnk{j}{n}{j+1}:=\Wn{j+1}{n-1}\vee\bigvee\sb{i\in I}\bS{n\sb{i}}} If we let
  \w{d\sb{0}=g\sb{i}} on \w[,]{\bS{n\sb{i}}} and other face maps vanish there (and
  add the necessary degeneracies), we obtain a new simplicial space \w{\Wk{n}{j+1}}
  for which the elements \w{\{\alpha\sb{i}\}\sb{i\in I}}
  are no longer in \w[.]{K:=\Ker(\pi\sb{n-1}\var{n-1})} Note that
  because \w{g\sb{i}} lands in \w[,]{Z\sb{j}\W{n-1}} it represents a permanent cycle
  in the spectral sequence of \S \ref{shssso} (see Remark \ref{rextaso}), so no
  elements of \w{\pi\sb{n+k-1}\|\W{n-1}\|} have been killed by this process.

  Proceeding in this way \wh with \w{\Wk{n}{s+1}} obtained from
  \w{\Wk{n}{s}} in stage $s$ by killing the elements in $K$ in filtration $s$ \wh
  we obtain \w{\Wk{n}{n}} with \w{\pi\sb{n-1}\tvar{n}} injective (there is
  no need to proceed beyond filtration $n$ for dimension reasons).

  We then wedge on $n$-spheres in simplicial dimension $0$ as in Case 1
  to obtain \w[.]{\W{n}}
\end{enumerate}
\end{proof}

The following fact will be used below and in \cite{BBSenA}:

\begin{prop}\label{pnoratho}
Let \w{\Wd} be a simplicial space in which each \w{\bW\sb{n}} is weakly equivalent to a
wedge of simply-connected rational spheres. Then the homotopy spectral sequence
for \w{\Wd} collapses at the \ww{E\sp{2}}-term.
\end{prop}

\begin{proof}
Using the differential graded Lie model of \cite{QuiR} for \w[,]{\Wd}  we see that
each \w{\bW\sb{n}} is \emph{coformal} (that is, has a cofibrant model with
$0$ differential).  This implies that every \w[,]{[\alpha]\in E\sp{2}\sb{n,k}}
represented by a Moore cycle \w[,]{\alpha\in Z\sb{n}\pi\sb{k}\Wd} is in fact
represented by \w{\hat{f}:S\sp{k}\to Z\sb{n}\Wd} \wwh so that it fits into a diagram
of the form \wref[.]{eqfulldata}
\end{proof}

\begin{remark}\label{rnoratho}
Of course, when \w{\bX=\|\Wd\|} itself is not coformal, the induced map
\w{f:\Sigma\sp{r}\bS{s}\to\bX} in \S \ref{dshoho} can still be non-trivial, so that
\w{[f]\in\pi\sb{r+s}\bX} may be the value of a higher homotopy operation, such as
a rational higher Whitehead product (see \cite{AArkS}).
\end{remark}

%
%
\sect{Atomic maps}
\label{catom}

Now that we have a procedure for generating new elements in \w{\pis\bX} by higher order
operations, we would like to show which elements are indecomposable with respect to such
operations.

\begin{thm}\label{tjct}
  If  $\bX$ is $k$-connected for \w[,]{k\geq 1} any homotopy class
  \w{\varphi\in \pi\sb{n}\bX} is generated of higher order over \w{\Pi\sb{>1}}
  (\S \ref{dgho}) by the atomic maps (\S \ref{datomicc}).
\end{thm}

\begin{proof}
We may assume that all spaces are localized at a prime \w[,]{p\geq 0} and consider the
sequential approximations \w{\W{m}} \wb{k\leq m\leq n} for $\bX$ constructed
in the proof of Proposition \ref{pseqapp}, which we may also take to be
$p$-local and $k$-connected in each simplicial dimension. Assume that $\varphi$
is not atomic, so that \w{\varphi\sb{\ast}:\HiF{n}{\bS{n}}\to\HiF{n}{\bX}}
is trivial.  By replacing $\bX$ by its \wwb{n+1}skeleton, we may assume it is
\wwb{n+1}dimensional (without affecting \w[).]{\pi\sb{n}\bX}

When \w{p>0} and \w{\varphi\sb{\ast}:\HiZp{n}{\bS{n}}\to\HiZp{n}{\bX}} is non-zero, then
the Hurewicz image of $\varphi$ is necessarily divisible by \w{p\sp{r}}
for some \w[.]{r\geq 1}

Now if \w{\bZ\,\xra{\tau}\,\hX\,\xra{i}\,\bX\,\xra{q}\,\Sigma\bZ} is the cofibration
sequence for some CW structure on $\bX$ with $n$-skeleton $\hX$,
with \w{\bZ:=\bigvee\sb{i=1}\sp{m}\,\bS{n}\lolr{i}} (here \w{\lra{i}} is just an index).
We have the associated Hurewicz diagram:
\mydiagram[\label{eqhurewicznsk}]{
  \pi\sb{n}\bZ \ar[r]\sp(0.6){\tau\sb{\#}} \ar[d]\sp{h\sb{n}}\sb{\cong} &
  \pi\sb{n}\hX \ar@{->>}[r]\sp{i\sb{\#}} \ar[d]\sp{h\sb{n}} &
  \pi\sb{n}\bX \ar[d]\sp{h\sb{n}}\ar[r] & \pi\sb{n}\Sigma\bZ=0 \ar[d]\\
\HiZp{n}\bZ \ar[r]\sp{\tau\sb{\ast}} &
\HiZp{n}{\hX} \ar@{->>}[r]\sp{i\sb{\ast}}  & \HiZp{n}{\bX} \ar[r] & \HiZp{n}{\Sigma\bZ}=0
}
\noindent Since \w{i\sb{\#}} is surjective, a representative $f$ for $\varphi$
lifts to \w{f':\bS{n}\to\hX} (cellular approximation), and since \w{h\sb{n}([f])}
is non-trivial by assumption, so is \w[.]{h\sb{n}([f'])}
However, we assumed \w{h\sb{n}([f])=p\sp{r}\cdot\hat{y}} for some
\w{\hat{y}\in\HiZp{n}{\bX}} and \w[,]{r\geq 1} so there is \w{y\in\HiZp{n}{\hX}} with
\w[.]{i\sb{\ast}(y)=\hat{y}} This means that
\w[,]{p\sp{r}\cdot y-h\sb{n}([f'])\in\Ker(i\sb{\ast})=\Image(\tau\sb{\ast})}
so if we write this expression as \w{\tau\sb{\ast}(h\sb{n}(\alpha))} for some
\w[,]{\alpha\in\pi\sb{n}\bZ} we see that we may replace our choice \w{f'} by
\w{f'':=f'+\tau\sb{\#}(\alpha)} with \w[,]{i\sb{\#}(f'')=f}  but now
\w{h\sb{n}(f'')} itself divisible by \w{p\sp{r}} in \w[.]{\HiZp{n}{\bX}}

Now assume \w[,]{p\geq 0} and let $\hX$ be (an $n$-skeleton of) \w[,]{\|\W{n}\|}
by definition (see \S \ref{sccwa}). Thus $\varphi$ factors through
\w[,]{f\up{n}:\bS{n}\to\|\W{n}\|} and we may assume that this is not atomic, either.
The class \w{[f\up{n}]\in\pi\sb{n}\|\W{n}\|} can be written
(non-uniquely) as a sum of elements in various filtrations. All those in positive
filtration are canonical values of higher order simplicial operations, by Proposition
\ref{pfiltrationn}. The elements in filtration $0$ factor through \w[,]{\Wn{0}{n}}
which is a wedge of spheres of dimensions \www[.]{\leq n} Thus by Hilton's Theorem
(see \cite{HilH}), each such element  is a sum of compositions of iterated Whitehead
products on the fundamental classes of these spheres. All such summands are decomposable,
except possibly for those which factor through the sub-wedge product \w{\bW'}
consisting of the $n$-spheres.  However, if
\w{f\up{n}\sb{\ast}:\HiZp{n}{\bS{n}}\to\HiZp{n}{\|\W{n}\|}} is non-zero and \w[,]{p>0}
then (because \w{f\up{n}} is not atomic) its Hurewicz image is $p$-divisible,
and since the Hurewicz map for \w{\bW'} is an isomorphism in dimension $n$,
this means the map \w{f\up{n}} is itself $p$-divisible, and thus decomposable.
For \w[,]{p=0} any non-zero map
\w{f\up{n}\sb{\ast}:\HiQ{n}{\bS{n}}\to\HiQ{n}{\|\W{n}\|}} is atomic.
\end{proof}

%
%
\sect{Atomic maps for spheres}
\label{cambs}

There is little hope of describing all atomic maps for general $\bX$, since the
Hurewicz image is not known even stably, in general (see \cite[\S 2]{JCohDS}).
However, we can do so for spheres: since the construction of \S \ref{sccwa}
is trivial for \w{\bS{k}} itself, more precisely, we consider a connected cover.
All spaces in this section are localized at a prime $p$.

\begin{thm}\label{ttransg}
For \w[,]{k\geq 2} let \w{\bX} be the homotopy fiber of the fundamental class
\w{\vare\sb{k}:\bS{k}\to\KZp{k}} (that is, the \wwb{r-1}connected cover of the $k$-sphere
for \w[),]{r:=k+2p-3} and assume that the lift \w{f:\bS{n}\to\bX} of \w{\varphi:\bS{n}\to\bS{k}}
is atomic for $p$. Then the homotopy cofiber of $\varphi$  supports a non-trivial mod $p$
cohomology operation.
\end{thm}

\begin{proof}
Our hypothesis is  that the image of \w{[f]\in\pi\sb{n}\bX} under the Hurewicz
homomorphism \w{\pi\sb{n}\bX\to\HiF{n}{\bX}} is non-trivial.
There is thus a class \w{\lambda\in\HuF{n}{\bX}} with \w[.]{f\sp{\ast}(\lambda)\neq 0}
Such a $\lambda$ is necessarily indecomposable in the unstable cohomology algebra
\w{\HuFs{\bX}} (that is, it does not decompose in terms of any unstable
cohomology operations, including the cup product).

We first observe that it suffices to show that any such indecomposable class $\lambda$ is
transgressive in the $\Fp$-cohomology Serre spectral sequence of the fibration sequence
\w[.]{\bX\hra\bS{k} \to\KZp{k}}

If we write \w{g:\bX\to\bS{k}} for the covering map, then \w{\varphi:=g\circ f} fits into
a commuting diagram with horizontal homotopy cofibration sequences, as follows:
\mysdiag[\label{eqcofseqs}]{
  \bS{n} \ar[d]^{f}\ar[rr]^{\varphi} && \bS{k} \ar[d]^{=}\ar[rr]^{j} && \Cof{\varphi}
  \ar[d]^{u}\ar[rr]^{\delta} && \bS{n} \ar[d]^{\Sigma f}\\
  \bX \ar[rr]^{g} && \bS{k} \ar[rr]^{s} \ar[rd]^{\vare\sb{k}} && \Cof{g} \ar[rr]^{\bdel}
  \ar[ld]^{\rho} && \Sigma\bX\\
  &&&  \KZp{k} &&&
}
\noindent where $t$ represents \w[,]{\eps{k}} and the map $\rho$ exists by the universal
properties of \w{\Cof{g}} and the fibration sequence \w[.]{\bX \xra{g}
\bS{k} \xra{\vare\sb{k}} \KZp{k}}

Since by assumption \w{f\sp{\ast}\lambda=\eps{n}} (the fundamental
class in \w[),]{\HuF{n}{\bS{n}}} under suspension we have
\w{\blam\in\HuF{n+1}{\Sigma\bX}} with \w[.]{(\Sigma
f)\sp{\ast}\blam=\eps{n+1}\in\HuF{n+1}{\bS{n+1}}}

If $\lambda$ is transgressive, there is a class \w{c\in
\HuF{n+1}{\KZp{k}}} with \w[.]{\rho\sp{\ast}c=\bdel\sp{\ast}\blam}
But $c$ is necessarily of the form \w{\theta(\iot{k})} for
\w{\iot{k}\in\HuF{k}{\KZp{k}}} the fundamental class, and a
diagram chase in:
\mywdiag[\label{eqccofseqs}]{
H\sp{k}{\bS{k}}\ni \eps{k}  & a\in H\sp{k}{\Cof{\varphi}}
\ar[l]^{j\sp{\ast}}  \ar@/^2em/@{-->}[r]^{\theta} &
H\sp{n+1}{\Cof{\varphi}}\ni b
& \ar[l]_{\delta\sp{\ast}} \eps{n+1}\in H\sp{n+1}{\bS{n+1}} \\
H\sp{k}{\bS{k}}\ni \eps{k} \ar[u]^{=} & & H\sp{n+1}{\Cof{g}}\ni
\bdel\sp{\ast}(\blam) \ar[u]^{u\sp{\ast}} &
\blam\in H\sp{n+1}{\Sigma\bX} \ar[l]_(0.4){\bdel\sp{\ast}} \ar[u]^{(\Sigma f)\sp{\ast}}\\
& \iot{k}\in H\sp{k}{\KZp{k}} \ar[ul]_{t\sp{\ast}}
\ar@/_2em/@{-->}[r]^{\theta}& c\in H\sp{n+1}{\KZp{k}}
\ar[u]^{\rho\sp{\ast}} }
\noindent shows that \w[,]{b=\theta(a)} where \w{a\in
\HuF{k}{\Cof{\varphi}}} has \w[,]{j\sp{\ast}a=\eps{k}} and \w{b\in
\HuF{n+1}{\Cof{\varphi}}} has \w[.]{\delta\sp{\ast}\eps{n+1}=b}
This $\theta$ is the non-trivial cohomology operation required by
our Theorem. We have thus reduced the proof of the Theorem to
showing that any indecomposable class $\lambda$ is transgressive
in the Serre spectral sequence for \w[.]{\bX\to\bS{k} \to\KZp{k}}

We compute \w[,]{\HuFs{\bX}} using the associated fibration sequence
\w[.]{\KZp{k-1} \xra{u}\bX\to\bS{k}} Since the base is a sphere, the Serre
spectral sequence reduces to the Wang long exact sequence of \cite{HCWangHF}:
\myrdiag[\label{spseqX}]{
\cdots\HuF{n}{\KZp{k-1}} \ar[rr]^{d\sb{k}} & &\HuF{n-k+1}{\KZp{k-1}} \ar[rr]^(0.66){\rho} &&
  \HuF{n+1}{\bX}\ar[lllld]\sb{\us} \\
 \HuF{n+1}{\KZp{k-1}} \ar[r] & \cdots &&
}
\noindent  We write \w{\Euz{p}} for the image of  \w{\HuFs{\bX}} in
\w{\HuFs{\KZp{k-1}}} \wwh that is, \w[\vsn.]{\Euz{p}:=\us\HuFs{\bX}}

From here on we distinguish between two cases\vsm:

 \noindent\textbf{Case I: The prime $2$\vsm.}

 For \w[,]{p=2} we have:
\begin{myeq}\label{eqemcoht}
\HuTs{\KZt{k}} \cong\Ft[\Sq{I}(\iot{k})~|\ I \mbox{admissible}, i\sb{s}\neq 1, \exc(I)<k].
\end{myeq}
\noindent (see  \cite[Proposition 3.5.8]{Koc}). Here the multi-index $I$ is of the form
\w{(i\sb{0},\dotsc,  i\sb{s})} with \w[,]{\Sq{I}:=\Sq{i\sb{0}}\dotsc\Sq{i\sb{s}}} and
\w{\exc(I)} is the excess.

The exact sequence \wref{spseqX} is determined by the fact that \w{d\sb{k}} is an
(anti-) derivation, which sends \w{\iot{k-1}} to $1$ and \w{\Sq{I}(\iot{k-1})} to $0$
for \w[.]{I\neq 0} It follows that
\begin{myeq}\label{eqemett}
\Euz{2}~=~\Ft[\iot{k-1}\sp{2},\ \Sq{I}(\iota_{k-1}) ~|\ 0\neq I \mbox{ admissible}, \
  i\sb{s}\neq1, \ \exc(I)<k-1]~.
\end{myeq}
\noindent We choose generators in \w{\HuTs{\bX}} corresponding to the polynomial algebra
generators \w{\iot{k-1}\sp{2}} and \w[,]{\Sq{I}{\iot{k-1}}} and also use the same notation
for them. Thus \w{\HuTs{\bX}} is generated as an \ww{\Euz{2}}-module by $1$ and
\w{\gam{2k-1}} \wb[,]{=\rho(\iot{k-1})} which lies in \w{\HuT{2k-1}{\bX}} under the exact
sequence \wref[).]{spseqX}
Therefore, the indecomposable classes in \w{\HuTs{\bX}} consist of \w{\gam{2k-1}}
and \w{\Sq{j}(\iot{k-1})} for \w[.]{j\leq k-1} Note that the notation
\w{\Sq{j}(\iot{k-1})} for a class in \w{\HuTs{\bX}} which is sent to the usual
\w{\Sq{j}(\iot{k-1})} in \w{\HuTs{\KZt{k}}} is not meant to imply that it is
in the image of a cohomology operation. 

We now verify that all these classes are transgressive in the \ww{\Ft}-cohomology Serre
spectral sequence for \w[,]{\bX\to\bS{k}\to\KZt{k}} using the following map $u$
of fibration sequences:
\mysdiag[\label{fibn}]{
 \KZt{k-1} \ar[d] \ar[r]^(0.65){u} & \bX \ar[d] \\
                   P\KZt{k} \ar@{->>}[d] \ar[r]  & \bS{k} \ar@{->>}[d]\sp{\eps{k}} \\
                   \KZt{k}  \ar[r]^{=}             & \KZt{k}\\
                   \cF{1} \ar[r]\sp{u} & \cF{2}
}
\noindent where the left hand side is the path-loop fibration.
This induces a map \w{\us} of the corresponding Serre spectral sequences, with
\w[.]{\Eis{r}{\us }:\Eis{r}{\cF{2}} \to\Eis{r}{\cF{1}}}
%
%
\begin{figure}[htbp]

\begin{center}
  \pgfsetshortenend{3pt}
  \pgfsetshortenstart{3pt}
  \begin{tikzpicture}[scale=1.1,line width=1pt]
    \draw[help lines] (-6.3,-2.3) grid (7.3,10.3);

  {\draw[fill]
    (-5.5,-2.3) circle (0pt) node[below=-1pt] {$0$}
    (-4.5,-2.3) circle (0pt) node[below=-1pt] {$1$}
    (-3.5,-2.3) circle (0pt) node[below=-1pt] {$\cdot$}
    (-2.5,-2.3) circle (0pt) node[below=-1pt] {$\cdot$}
    (-1.5,-2.3) circle (0pt) node[below=-1pt] {$\cdot$}
    (-.5,-2.3) circle (0pt) node[below=-1pt] {$\cdot$}
    (.5,-2.3) circle (0pt) node[below=-1pt] {$k$}
    (1.5,-2.3) circle (0pt) node[below=-1pt] {$k+1$}
    (2.5,-2.3) circle (0pt) node[below=-1pt] {$k+2$}
    (3.5,-2.3) circle (0pt) node[below=-1pt] {$\cdot$}
   (4.5,-2.3) circle (0pt) node[below=-1pt] {$\cdot$}
   (5.5,-2.3) circle (0pt) node[below=-1pt] {$\cdot$}
   (6.5,-2.3) circle (0pt) node[below=-1pt] {$2k$}
    (.3,-2.8) circle (0pt) node[below=-1pt] {  $\HuTs{\KZt{k}} $}
    ;}

 {\draw[fill]
    (-6.6,-1.4) circle (0pt) node[below=-1pt] {$0$}
    (-6.6,-.4) circle (0pt) node[below=-1pt] {$1$}
    (-6.6,.6) circle (0pt) node[below=-1pt] {$\cdot$}
    (-6.6,1.6) circle (0pt) node[below=-1pt] {$\cdot$}
    (-6.6,2.6) circle (0pt) node[below=-1pt] {$\cdot$}
    (-6.6,3.6) circle (0pt) node[below=-1pt] {$k-1$}
    (-6.6,4.6) circle (0pt) node[below=-1pt] {$k$}
    (-6.6,5.6) circle (0pt) node[below=-1pt] {$k+1$}
    (-6.6,6.6) circle (0pt) node[below=-1pt] {$\cdot$}
    (-6.6,7.6) circle (0pt) node[below=-1pt] {$\cdot$}
   (-6.6,8.6) circle (0pt) node[below=-1pt] {$2k-2$}
   (-6.6,9.6) circle (0pt) node[below=-1pt] {$2k-1$}
    (-7.6,2.8) circle (0pt) node[below=-1pt,rotate=90] {  $\HuTs{\KZt{k-1}} $}
    ;}

 {\draw[fill]
    (-5.5,-1.4) circle (0pt) node[below=-1pt] {$1$}
    (.5,-1.4) circle (0pt) node[below=-1pt] {$\iot{k}$}
    (2.5,-1.4) circle (0pt) node[below=-1pt] {\scriptsize $\Sq{2}\iot{k}$}
    (5.5,-1.4) circle (0pt) node[below=-1pt] {\scriptsize $\Sq{k-1}\iot{k} $}
    (6.5,-1.4) circle (0pt) node[below=-1pt] {\scriptsize $\iot{k}^2$}
    (-5.5,3.6) circle (0pt) node[below=-1pt] {$\iot{k-1}$}
    (-5.5,5.6) circle (0pt) node[below=-1pt] {\scriptsize $\Sq{2}\iot{k-1}$}
    (-5.5,8.6) circle (0pt) node[below=-1pt] {\scriptsize $\iot{k-1}^2$}
    (.5,3.6) circle (0pt) node[below=-1pt] {\scriptsize $\iot{k-1}\iot{k}$}
    ;}

  {\draw[->]
(-5.3,3.4) -- (.4,-1.3) node[midway,left]{\scriptsize $d_k$}
;}

  {\draw[->]
(-5.1,5.2) -- (2.4,-1.3) node[midway,left]{\scriptsize $d_{k+2}$}
;}

{\draw[->]
(.3,3.4) -- (6.4,-1.3) node[midway,left]{\scriptsize $d_k$}
;}

{\draw[->,out=310,in=140]
(-5.3,8.4) to node[midway,left]{\scriptsize $d_{2k-1}$}  (5.4,-1.3)
;}
\end{tikzpicture}
\end{center}
\caption{Spectral sequence for the fibration sequence $\cF{1}$.}
\label{SpseqEM}
\end{figure}

Observe that the first differential for \w{\cF{1}} is \w[,]{d\sb{k}} determined by
\w[.]{d\sb{k}(\iot{k-1})= \iot{k}} (See Figure \ref{SpseqEM}).

The key step in the argument is the calculation of \w[.]{\KK{r}:= \Ker(\Eis{r}{\us})}
For \w[,]{j\leq k} observe that \w[.]{\KK{r}=\gam{2k-1}\cdot\Eis{r}{\cF{2}}}
For \w[,]{j=k+1} multiples of \w{\iot{k}} are zero in \w[,]{\cF{1}} so that
\w[.]{\KK{k+1}=\gam{2k-1}\cdot\Eis{k+1}{\cF{2}} + \iot{k}\cdot\Eis{k+1}{\cF{2}}}
Using our comparison map \w{\us:\Eis{r}{\cF{2}}\to\Eis{r}{\cF{1}}} and the fact
that \w{\Sq{j}(\iot{k-1})} transgresses to \w{\Sq{j}(\iot{k})} in \w[,]{\cF{1}}
we deduce that either the class \w{\Sq{j}(\iot{k-1})} transgresses to
\w{\Sq{j}(\iot{k})} in \w[,]{\cF{2}} or else some differential carries \w{\Sq{j}(\iot{k-1})}
to a class in  \w{\KK{r}} for some $r$. Now note that for \w[,]{2\leq j \leq k-1}
\w[,]{k+1\leq  |\Sq{j}(\iot{k-1})| \leq 2k-2} and that in these degrees \w{\KK{r}} is zero.
This forces \w{\Sq{j}(\iot{k-1})} to be transgressive.

The class \w{\gam{2k-1}} lies in \w{\KK{r}} for all $r$. As the spectral sequence converges
to the cohomology of \w[,]{\bS{k}} it must support a differential. Therefore,
a differential on it must also lie in \w[,]{\KK{r}} which leaves the only possibility
as \w[,]{d\sb{2k}(\gam{2k-1})=\iot{k}\sp{2}} for degree reasons.

%
%

\begin{figure}[htbp]
\begin{center}
  \pgfsetshortenend{3pt}
  \pgfsetshortenstart{3pt}
  \begin{tikzpicture}[scale=1.1,line width=1pt]
    \draw[help lines] (-6.3,-2.3) grid (7.3,10.3);
  {\draw[fill]
    (-5.5,-2.3) circle (0pt) node[below=-1pt] {$0$}
    (-4.5,-2.3) circle (0pt) node[below=-1pt] {$1$}
    (-3.5,-2.3) circle (0pt) node[below=-1pt] {$\cdot$}
    (-2.5,-2.3) circle (0pt) node[below=-1pt] {$\cdot$}
    (-1.5,-2.3) circle (0pt) node[below=-1pt] {$\cdot$}
    (-.5,-2.3) circle (0pt) node[below=-1pt] {$\cdot$}
    (.5,-2.3) circle (0pt) node[below=-1pt] {$k$}
    (1.5,-2.3) circle (0pt) node[below=-1pt] {$k+1$}
    (2.5,-2.3) circle (0pt) node[below=-1pt] {$k+2$}
    (3.5,-2.3) circle (0pt) node[below=-1pt] {$\cdot$}
   (4.5,-2.3) circle (0pt) node[below=-1pt] {$\cdot$}
   (5.5,-2.3) circle (0pt) node[below=-1pt] {$\cdot$}
   (6.5,-2.3) circle (0pt) node[below=-1pt] {$2k$}
    (.3,-2.8) circle (0pt) node[below=-1pt] {  $\HuTs{\KZt{k}}  $}
    ;}

 {\draw[fill]
    (-6.6,-1.4) circle (0pt) node[below=-1pt] {$0$}
    (-6.6,-.4) circle (0pt) node[below=-1pt] {$1$}
    (-6.6,.6) circle (0pt) node[below=-1pt] {$\cdot$}
    (-6.6,1.6) circle (0pt) node[below=-1pt] {$\cdot$}
    (-6.6,2.6) circle (0pt) node[below=-1pt] {$\cdot$}
    (-6.6,3.6) circle (0pt) node[below=-1pt] {$k-1$}
    (-6.6,4.6) circle (0pt) node[below=-1pt] {$k$}
    (-6.6,5.6) circle (0pt) node[below=-1pt] {$k+1$}
    (-6.6,6.6) circle (0pt) node[below=-1pt] {$\cdot$}
    (-6.6,7.6) circle (0pt) node[below=-1pt] {$\cdot$}
   (-6.6,8.6) circle (0pt) node[below=-1pt] {$2k-2$}
   (-6.6,9.6) circle (0pt) node[below=-1pt] {$2k-1$}
    (-7.6,2.8) circle (0pt) node[below=-1pt,rotate=90] {  $\Hus{\bX}{\mathbb{F}_2}$}
    ;}

 {\draw[fill]
    (-5.5,-1.4) circle (0pt) node[below=-1pt] {$1$}
    (.5,-1.4) circle (0pt) node[below=-1pt] {$\iot{k}$}
    (2.5,-1.4) circle (0pt) node[below=-1pt] {\scriptsize $\Sq{2}\iot{k}$}
    (5.5,-1.4) circle (0pt) node[below=-1pt] {\scriptsize $\Sq{k-1}\iot{k} $}
    (6.5,-1.4) circle (0pt) node[below=-1pt] {\scriptsize $\iot{k}^2$}
    (-5.5,5.6) circle (0pt) node[below=-1pt] {\scriptsize $\Sq{2}\iot{k-1}$}
    (-5.5,8.6) circle (0pt) node[below=-1pt] {\scriptsize $\iot{k-1}^2$}
   (-5.5,9.6) circle (0pt) node[below=-1pt] {\scriptsize $\gamma_{2k-1}$}
    ;}

  {\draw[->]
(-5.1,5.2) -- (2.4,-1.3) node[midway,left]{\scriptsize $d_{k+2}$}
;}

{\draw[->]
(-5.3,9.4) -- (6.4,-1.3) node[midway,left]{\scriptsize $d_{2k}$}
;}

{\draw[->] (-5.3,8.4) to node[midway,left]{\scriptsize $d_{2k-1}$}
(5.4,-1.3) ;}
\end{tikzpicture}
\end{center}

  \caption{Spectral sequence for the fibration sequence $\cF{2}$}
 \label{SpseqX}
\end{figure}

\vsm\quad

 \noindent\textbf{Case II: odd primes\vsm.}

For an odd prime $p$, we have:
\begin{myeq}\label{eqemcohp}
  \HuFs{\KZp{k}} \cong F\sp{\gr}\sb{\Fp}[\cP{I}(\iot{k})~|\
    \mbox{ admissible}, \eps{s+1}(I)=0, \exc(I)<k-1]
\end{myeq}
  \noindent (see \cite[Proposition 3.5.8]{Koc}), where now
  \w{I:=(\eps{0},i\sb{0},\cdots, \eps{s}, i\sb{s}, \eps{s+1})} for \w[,]{\eps{i}\in\{0,1\}}
  and  \w[.]{\cP{I} = \beta\sp{\eps{0}}\cP{i_0}\cdots\cP{i\sb{s}}\beta\sp{\eps{s+1}}}
  Here \w{F\sp{\gr}\sb{\Fp}} denotes the free graded commutative algebra \wh
  that is, exterior on the odd degree classes and polynomial on the even
  degree classes\vsn.

  We now distinguish between two cases, depending on the parity of  $k$:
when $k$ is odd, we find
\begin{myeq}\label{eqemettodd}
  \Euz{p}=F\sp{\gr}\sb{\Fp}[\iot{k-1}\sp{p},\,
    \cP{I}(\iot{k-1})~|\ I\neq 0 \mbox{ admissible}, \eps{s+1}(I)=0, \exc(I)<k]
\end{myeq}
\noindent Note that \w{d\sb{k}(\iot{k-1}\sp{p})=0} in the exact sequence \wref[,]{spseqX} while
\w{d\sb{k}(\iot{k-1}\sp{j})=j\iot{k-1}\sp{j-1}} for \w[.]{j\leq p-1}  As in the \w{p=2} case,
we write down the corresponding generators in \w{\HuFs{\bX}} using the same notation.
It follows that \w{\HuFs{\bX}} is generated as an \ww{\Euz{p}}-module by $1$ and
\w[,]{\gam{(k-1)p+1}} which is the image of
\w{\iot{k-1}\sp{p-1}} in \w{\HuF{(p-1)(k-1)+k}{\bX}} in the exact sequence \wref[.]{spseqX}

Therefore, the indecomposable classes in \w{\HuFs{\bX}} are \w[,]{\gam{(k-1)p+1}} and
\w{\cP{j}(\iot{k-1})} for \w{j\leq \frac{k-1}{2}} (noting that
\w[).]{\iot{k-1}\sp{p} =\cP{\frac{k-1}{2}}(\iot{k-1})}
We verify that all these classes are transgressive in the \ww{\Fp}-cohomology Serre
spectral sequence for the fibration sequence \w[.]{\bX\to\bS{k} \to\KZp{k}}

This follows from a very similar argument to that for the case \w[.]{p=2}
Note, however, that for reasons of degree, the classes \w{\cP{i}(\iot{k-1})} must be
transgressive, while the class \w{\gam{(k-1)p+1}} must transgress to the class
\w{\beta\cP{\frac{k-1}{2}}(\iot{k})} (which is the differential on
\w{\iot{k-1}\sp{p-1}\cdot\iot{k}} in the spectral sequence \w[]{\cF{1}}
by the Kudo transgression theorem (see \cite[Theorem 3.5.3]{Koc})\vsn .

For $k$ even, we have
\begin{myeq}\label{eqemetteven}
  \Euz{p} =  F\sp{\gr}_{\Fp}[\cP{I}(\iot{k-1})~|\ I\neq 0 \mbox{ admissible},
    \eps{s+1}(I)=0, \exc(I)<k-1]~,
\end{myeq}
\noindent as \w[.]{\iot{k-1}\sp{2}=0} It follows that \w{\HuFs{\bX}} is generated as
an \ww{\Euz{p}}-module by $1$ and \w[,]{\gam{2k-1}} which is the image of
\w{\iot{k-1}} in \w{\HuF{2k-1}{\bX}} in the exact sequence \wref[.]{spseqX}
  Therefore, the indecomposable classes in \w{\HuFs{\bX}} are \w[,]{\gam{2k-1}}  and
  \w{\cP{j}(\iot{k-1})} for \w[.]{j< \frac{k-1}{2}}
  As above by calculating \w{\KK{r}} we see that the only class that \w{d\sb{r}(\gam{2k-1})}
  can be is \w[,]{\iot{k}\sp{2}} so it is transgressive.

  Finally for the classes \w[,]{\cP{j}(\iot{k-1})} observe that if these classes are
  \emph{not} transgressive, the differential on these classes must lie in \w{\KK{r}}
  for some $r$ (as they are transgressive in the spectral sequence for \w[).]{\cF{1}}
Observe that \w{\KK{r} =\{\gam{2k-1}\Eis{r}{\cF{2}}} for \w[,]{r\leq k} while
\w[.]{\KK{k+1} = \gam{2k-1}\cdot\Eis{k+1}{\cF{2}}+ \iot{k}\cdot\Eis{k+1}{\cF{2}}}

If \w{d\sb{r}(\cP{j}(\iot{k-1})\neq 0} for some \w[,]{r< 2k} we must have \w{r > k} and
\w[.]{d\sb{r}(\cP{j}(\iot{k-1}))=\gam{2k-1}q} Note that $q$ must be in the $r$-th column,
and hence is of the form \w{\cP{L}(\iot{k})q'} with \w[,]{\exc(L)<k-1} where \w{q'} is
in the $0$-th column, since \w[.]{r<2k} Thus
\w[.]{d\sb{r}(\cP{L}(\iot{k-1})q'\gam{2k-1})=\gam{2k-1}q}

Observe that the classes \w{\cP{j}(\iot{k-1})} in the cohomology of $X$ are defined only
up to a multiple of \w[,]{\gam{2k-1}} so we may change the representative in such a way
that the differential \w{d\sb{r}} vanishes on it. Finally, note that \w{\KK{2k+1}} is
simply \w[,]{\Fp\{\iot{k}\}} so from this page on the differentials are determined by
those of  the spectral sequence for \w[.]{\cF{1}} The result follows.
\end{proof}

\begin{corollary}\label{catomic}
All homotopy classes of maps between wedges of simply-connected spheres are
generated of higher order by the fundamental classes, the Hopf maps
\w[,]{\eta\sb{k}\in\pi\sb{k+1}\bS{k}}
\w[,]{\nu\sb{k}\in\pi\sb{k+3}\bS{k}} and
\w[,]{\sigma\sb{k}\in\pi\sb{k+7}\bS{k}} and the maps
\w[,]{\alpha\sb{1}(p)\in\pi\sb{k+2p-3}\bS{k}} for odd primes $p$ and \w[.]{k\geq 2}
\end{corollary}

\begin{proof}
This follows from the stable analysis in \cite{AdHI} and \cite{LiuFC}.
Note that we may post-compose the augmentation obtained from a resolution
\w{\Wd} of \w{\|\Wd\|\simeq\bS{k}\lra{k}} with the covering map
\w{p:\bS{k}\lra{k}\to\bS{k}} to obtain the augmentation \w{\Wd\to\bX=\bS{k}}
in Definition \ref{dshoho}: this shows
that if \w{\alpha\in\pis\bS{k}\lra{k}} is decomposable (that is, a value for a
higher order operation associated to \w[),]{\Wd} then \w{p\sb{\#}\alpha} is, too.
\end{proof}

In other words, the ``module'' decomposability in \w{\pis\bS{k}\lra{k}}
in the sense of Definition \ref{dgho} implies the ``algebra'' decomposability
in \w{\pis\bS{k}} itself in the sense of Definition \ref{dgoho}.

%
%
\sect{Suspensions}
\label{csusp}

So far we have only considered the way elements in the homotopy groups of a single
space $\bY$ can be decomposed. However, our methods also allow us to describe
the elements generated by suspending $\bY$.

\begin{mysubsection}{The suspension spectral sequence}
\label{ssuspss}
Given a connected pointed space $\bY$, we denote by \w{\bY\otimes S\sp{n}}
the $n$-skeletal simplicial space with a single non-degenerate copy of $\bY$
in simplicial dimension $n$, so that its geometric realization is weakly equivalent
to \w[.]{\Sigma\sp{n}\bY}

In particular, the homotopy spectral sequence for \w{\Wd:=\bY\otimes S\sp{1}} will
be called the \emph{suspension spectral sequence} for $\bY$. This simplicial space
may be described combinatorially in terms of the degeneracies of $\bY$:
\mytdiag[\label{eqsupssp}]{
  \dotsc s\sb{1}s\sb{0}\bY\vee s\sb{2}s\sb{0}\bY\vee s\sb{2}s\sb{1}\bY
  \ar@<3ex>[rr]\sp(0.6){d\sb{0}} \ar@<1ex>[rr]\sp(0.6){d\sb{1}}
  \ar@<-1ex>[rr]\sp(0.6){d\sb{2}}
  \ar@<-2ex>[rr]\sb(0.6){d\sb{3}} &&
  s\sb{0}\bY\vee s\sb{1}\bY
    \ar@<2ex>[rr]\sp(0.5){d\sb{0}} \ar[rr]\sp(0.5){d\sb{1}}
  \ar@<-1ex>[rr]\sb(0.5){d\sb{2}} && \bY
 \ar@<0.5ex>[r]\sp(0.6){d\sb{0}}\ar@<-0.5ex>[r]\sb(0.6){d\sb{1}} &
 \ast\\
 \dotsc \bW\sb{3} && \bW\sb{2} && \bW\sb{1} & \bW\sb{0}
}
\noindent with the degeneracies as indicated, and all face maps determined by the
simplicial identities.

Since \w[,]{\pis\bW\sb{0}=0} \w{\pis\Sigma\bY} has no elements in filtration $0$ in
this spectral sequence, so all its elements are formally exhibited as values
of higher simplicial homotopy operations.
\end{mysubsection}

\begin{remark}\label{rsusp}
The elements of \w{\pis\Sigma\bY} in filtration $1$ are evidently
just the image of the suspension homomorphism
\w[.]{E:\pis\bY\to\pi\sb{\ast+1}\Sigma\bY} Note that it is natural to think of
suspensions as values of a secondary homotopy operation, since any map
\w{f:\bS{n}\to\bY} is null homotopic in \w{\Sigma\bY} for two different reasons \wh
namely, the two cones on the ``equatorial'' copy of $\bY$.
\end{remark}

\begin{example}\label{egnup}
In the spectral sequence for \w[,]{\Wd:=\bS{2}\otimes S\sp{1}}
the $1$-cycles \w[,]{\iot{2}\in\pi\sb{2}\bW\sb{1}}
\w[,]{\eta\sb{2}\in\pi\sb{3}\bW\sb{1}}  and
\w{\eta\sb{2}\eta\sb{3}\in\pi\sb{4}\bW\sb{1}} represent the suspension classes
\w[,]{\iot{3}\in\pi\sb{3}\bS{3}}
\w[,]{\eta\sb{3}\in\pi\sb{4}\bS{3}}  and
\w[,]{\eta\sb{3}\eta\sb{4}\in\pi\sb{5}\bS{3}} respectively.

The $2$-chain \w{[s\sb{0}\iot{2},s\sb{1}\iot{2}]\in\pi\sb{3}\bW\sb{2}}
is not a cycle, since \w{[\iot{2},\iot{2}]\neq 0} in \w[.]{\pi\sb{3}\bW\sb{1}}
However, \w{[s\sb{0}\iot{2},s\sb{1}\iot{2}]\circ\eta\sb{3}\in\pi\sb{4}\bW\sb{2}}
is a cycle, since
\w{[\iot{2},\iot{2}]\circ\eta\sb{3}=(2\eta\sb{2})\circ\eta\sb{3}=0} in
\w[.]{\pi\sb{4}\bW\sb{1}} It has order $2$, and is not a boundary, so it must represent
\w[,]{\pm\nu'\in\pi\sb{6}\bS{3}} with \w{2\nu'=\eta\sb{3}\eta\sb{4}\eta\sb{5}}
(a suspension, and thus in filtration $1$).
\end{example}

\begin{remark}\label{rgsusp}
More generally, we could start with any simplicial resolution \w{\Wd} of $\bY$, and by
letting \w{\Vd:=\Wd\otimes S\sp{k}} we obtain a resolution of \w[.]{\Sigma\sp{k}\bY}
\end{remark}

\begin{mysubsection}{Whitehead products}
\label{swhite}
To identify certain terms in higher filtration in the suspension spectral sequence,
we consider the special case where \w[.]{\bY=\bS{p}\vee\bS{q}} Note that we have
two split inclusions \w{i\sb{p}:\bS{p}\hra\bY} and \w[,]{i\sb{q}:\bS{q}\hra\bY}
so the spectral sequences for each summand split off from that for $\bY$; we are only
interested in the remaining cross-term part.

First assume that \w[:]{p,q\geq 2} by Hilton's Theorem, we know
that the lowest dimensional non-trivial cross-term cycle in the simplicial abelian
group \w{\pi\sb{i}\Wd} in filtration (=simplicial dimension) $2$ necessarily
has the form:
\begin{myeq}\label{eqwhct}
w\sb{2}~:=~[s\sb{0}i\sb{p},\,s\sb{1}i\sb{q}]~-~[s\sb{1}i\sb{p},\,s\sb{0}i\sb{q}]~\in~
\pi\sb{p+q-1}\bW\sb{2}
\end{myeq}
\noindent (see \wref[).]{eqsupssp} For dimensional reasons this cycle cannot be hit
by any differentials, and the class it represents in
\w{\pi\sb{p+q+1}\|\Wd\|=\pi\sb{p+q+1}(\bS{p+1}\vee\bS{q+1})} is therefore the
generator \w{[i\sb{p+1},\,i\sb{q+1}]} (up to sign) \wh which thus can be expressed as
the value of a third order homotopy operation.

Likewise, the lowest dimensional cycles in filtration $3$ will have the form
\begin{myeq}\label{eqwhcth}
\begin{split}
&~[s\sb{1}s\sb{0}i\sb{p},[s\sb{2}s\sb{0}i\sb{p},\,s\sb{2}s\sb{1}i\sb{q}]]
  -[s\sb{1}s\sb{0}i\sb{p},[s\sb{2}s\sb{1}i\sb{p},\,s\sb{2}s\sb{0}i\sb{q}]]
  +[s\sb{2}s\sb{0}i\sb{p},[s\sb{2}s\sb{1}i\sb{p},\,s\sb{1}s\sb{0}i\sb{q}]]\\
  &-[s\sb{2}s\sb{0}i\sb{p},[s\sb{1}s\sb{0}i\sb{p},\,s\sb{2}s\sb{1}i\sb{q}]]
  +[s\sb{2}s\sb{1}i\sb{p},[s\sb{1}s\sb{0}i\sb{p},\,s\sb{2}s\sb{0}i\sb{q}]]
  -[s\sb{2}s\sb{1}i\sb{p},[s\sb{2}s\sb{0}i\sb{p},\,s\sb{1}s\sb{0}i\sb{q}]]
\end{split}
\end{myeq}
\noindent representing \w{[i\sb{p+1},[i\sb{p+1},i\sb{q+1}]]} in
\w{\pi\sb{2p+q+1}\in\bS{p+1}\vee\bS{q+1}} (and similarly for the other Hall basis
elements).  Note that the $2$-cycle
\w[,]{[s\sb{0}i\sb{p},[s\sb{0}i\sb{p},\,s\sb{1}i\sb{q}]]
  -[s\sb{1}i\sb{p},[s\sb{0}i\sb{p},\,s\sb{1}i\sb{q}]]} in the cross-term of
\w[,]{\pi\sb{2p+q-2}\bW\sb{2}} bounds
\w[,]{[s\sb{2}s\sb{0}i\sb{p},[s\sb{1}s\sb{0}i\sb{p},\,s\sb{2}s\sb{1}i\sb{q}]]}
and so on.

When \w[,]{p=q=1} \w{\pis\Wd} is concentrated in dimension $1$, and the
Whitehead products are commutators, so we cannot use a dimension argument to
identify the cycles.  In this case it is convenient to work with a simplicial group
model for \w[,]{\bS{1}\vee\bS{1}} namely, \w{F(S\sp{0}\vee S\sp{0})} \wwh a constant
(free) simplicial group. Thus \w{\Wd} is a bisimplicial group which is equivalent to its
own diagonal, which is thus a simplicial group model \w{F(S\sp{1}\vee S\sp{1})}for
\w[,]{S\sp{2}\vee S\sp{2}} and the spectral sequence collapses at the
\ww{E\sp{2}}-term.

In this case we see that the $1$-dimensional analogue of \wref{eqwhct} is the
product of two commutators:
\begin{myeq}\label{eqwhco}
  \omega~:=~
  \left(s\sb{0}\alpha s\sb{1}\beta s\sb{0}\alpha\sp{-1} s\sb{1}\beta\sp{-1}\right)
\cdot\left(s\sb{1}\alpha s\sb{0}\beta s\sb{1}\alpha\sp{-1} s\sb{0}\beta\sp{-1}\right)~,
\end{myeq}
\noindent where \w{\alpha\in\pi\sb{1}\bS{1}\lo{1}} and
\w{\beta\in\pi\sb{1}\bS{1}\lo{2}} are the fundamental classes of the two wedge
summands in \w[.]{\bY=\bS{1}\lo{1}\vee\bS{1}\lo{2}} This is known to represent the
Whitehead product in \w{S\sp{2}\vee S\sp{2}} (see \cite[\S 11.11]{CurtS}).
A similar statement holds for the case \w[.]{p=1<q}

Since \w{\bS{k}\vee\bS{k}} is the universal space for Whitehead products of
$k$-dimensional classes, we see that whenever $\bY$ is \wwb{k-1}connected
and \w[,]{\alpha,\beta\in\pi\sb{k}\bY} the Whitehead product \w{[E\alpha,E\beta]}
in \w{\pi\sb{2k+1}\Sigma\bY} is represented in the spectral sequence by
\begin{myeq}\label{eqwhprodrep}
\gamma~:=~[s\sb{0}\alpha,\,s\sb{1}\beta]-[s\sb{1}\alpha,\,s\sb{0}\beta]~,
\end{myeq}
\noindent modulo suspensions.
\end{mysubsection}

We can use the suspension spectral sequence to provide an elementary proof of the
following well-known facts (cf.\ \cite[(3.5)]{HWhiW}):

\begin{lemma}\label{lsusp}
Let $\bY$ be \wwb{k-1}connected \wb[,]{k\geq 1} and \w[:]{\alpha,\beta\in\pi\sb{k}\bY}

\begin{enumerate}
\renewcommand{\labelenumi}{(\alph{enumi})~}
\item If \w{[\alpha,\beta]=0} in \w[,]{\pi\sb{2k-1}\bY}
then \w{[E\alpha,E\beta]} is divisible by $2$ in
\w[,]{\pi\sb{2k+1}\Sigma\bY} modulo elements in the image of the suspension
\w[.]{E:\pis\bY\to\pi\sb{\ast+1}\Sigma\bY}
\item If \w{\alpha=\beta} and $k$ is even, then \w{[E\alpha,E\beta]} is a suspension.
\end{enumerate}
\end{lemma}

\begin{proof}
  By the universal example \wref[,]{eqwhct} the class
  \w{\gamma\in\pi\sb{2}\pi\sb{2k-1}\Wd} of \wref{eqwhprodrep} is a cycle representing
  \w{[E\alpha,E\beta]} (modulo lower filtration, which is \w[).]{\Image(E)}
\begin{enumerate}
\renewcommand{\labelenumi}{(\alph{enumi})~}
\item If \w[,]{[\alpha,\beta]=0} then \w{\delta:=[s\sb{0}\alpha,\,s\sb{1}\beta]}
  is another cycle which survives to \w{\pis\Sigma\bY} whose double is $\gamma$.
\item When $k$ is even, \w[,]{\gamma=0} so \w{[E\alpha,E\beta]} vanishes modulo
  suspensions.
\end{enumerate}
\end{proof}

\begin{examples}\label{egsusp}

\begin{enumerate}
\renewcommand{\labelenumi}{(\alph{enumi})~}
\item When \w{\bY=\bS{1}} and \w[,]{\alpha=\beta=\iot{1}}
then \w{[\alpha,\beta]=0} in \w{\pi\sb{1}\bS{1}} and the Lemma yields the Hopf map
\w[,]{\eta\sb{2}\in\pi\sb{3}\bS{2}} with \w{2\eta\sb{2}=[\iot{2},\iot{2}]}
(since there are no elements in filtration $0$ in this dimension).
\item When \w{\bY=\bS{3}} and \w[,]{\alpha=\beta=\iot{3}}
again \w{[\alpha,\beta]=0} in \w{\pi\sb{5}\bS{3}} and the Lemma yields
\w[,]{\nu\sb{4}\in\pi\sb{7}\bS{4}} and in fact
\begin{myeq}\label{eqwhsqfour}
2\nu\sb{4}=[\iot{4},\iot{4}]+E\nu'
\end{myeq}
\noindent by \cite[(5.8)]{TodC}, although we cannot determine the extension
using the spectral sequence alone.
\item When \w{\bY=\bS{4}} and \w[,]{\alpha=\beta=\iot{4}} we have
\w[,]{[\iot{5},\iot{5}]=\nu\sb{5}\circ\eta\sb{8}=E(\nu\sb{4}\circ\eta\sb{7})}
where the first equality is \cite[(5.10)]{TodC}.
\item When \w{\bY=\bS{6}} and \w[,]{\alpha=\beta=\iot{6}} we have
\w[.]{[\iot{7},\iot{7}]=0}
\item When \w{\bY=\bS{7}} and \w[,]{\alpha=\beta=\iot{7}}
we have \w{[\alpha,\beta]=0} and the Proposition yields
\w[,]{\sigma\sb{8}\in\pi\sb{15}\bS{8}} with
\begin{myeq}\label{eqwhsqeight}
2\sigma\sb{8}=[\iot{8},\iot{8}]+E\sigma'
\end{myeq}
\noindent by \cite[(5.16)]{TodC}.
\item When \w{\bY=\bS{8}} and \w[,]{\alpha=\beta=\iot{8}} we have
$$
[\iot{9},\iot{9}]=
\sigma\sb{9}\circ\eta\sb{16}+\bar{\nu}\sb{9}+\epsilon\sb{9}=
E(\sigma\sb{8}\circ\eta\sb{15}+\bar{\nu}\sb{8}+\epsilon\sb{8})~,
$$
\noindent where the first equality is \cite[(7.1)]{TodC}.
\end{enumerate}
\end{examples}

\begin{remark}\label{rsigmappp}
As noted above, in the spectral sequence for \w[,]{\Wd:=\bS{4}\otimes S\sp{1}}
the $1$-chain
\w{\delta\sb{4}:=[s\sb{0}\iot{4},\,s\sb{1}\iot{4}]\in\pi\sb{7}\bW\sb{2}}
is not a cycle, since
\w{d\sb{0}\delta\sb{4}=[\iot{4},\iot{4}]=2\nu\sb{4}-E\nu'}
by \wref[.]{eqwhsqfour} Since \w{E\nu'\circ\eta\sb{7}\neq 0} and
\w[,]{E\nu'\circ\eta\sb{7}\circ\eta\sb{8}\neq 0} also \w{\delta\sb{4}\circ\eta\sb{7}}
and \w{\delta\sb{4}\circ\eta\sb{7}\circ\eta\sb{8}} are not cycles. Similarly,
$$
(2\nu\sb{4}-E\nu')\circ(2\nu\sb{7})=
4\nu\sb{4}\circ\nu\sb{7}-(2E\nu')\circ\nu\sb{7}=
4\nu\sb{4}\circ\nu\sb{7}-\eta\sp{3}\circ\nu\sb{7}=4\nu\sb{4}\circ\nu\sb{7}\neq 0~,
$$
\noindent since \w{2\nu'=\eta\sp{3}} and already \w[,]{\eta\sb{5}\circ\nu\sb{6}=0}
while \w{\nu\sb{4}\circ\nu\sb{7}} has order $8$ in\w[.]{\pi\sb{10}\bS{4}}
This implies that \w{2\delta\sb{4}\circ\nu\sb{4}} and thus
\w{\delta\sb{4}\circ\nu\sb{4}} are not cycles, but \w{\delta\sb{4}\circ\eta\sp{3}}
is a cycle in \w[.]{\pi\sb{10}\bW\sb{2}} Since it cannot bound anything,
it must therefore represent the generator \w{\sigma'''} of
\w[.]{\pi\sb{12}\bS{5}\lo{2}\cong\ZZ/2}

Similarly, in the spectral sequence for \w[,]{\Wd:=\bS{6}\otimes S\sp{1}}
\w{\delta\sb{6}:=[s\sb{0}\iot{6},\,s\sb{1}\iot{6}]\in\pi\sb{11}\bW\sb{2}}
is not a cycle, since
\w{d\sb{0}\delta\sb{6}~=~[\iot{6},\iot{6}]\in\pi\sb{11}\bS{6} \cong \ZZ}
has infinite order. However, \w{\delta\sb{6}\circ\eta\sb{11}} is the only possible
representative for \w[,]{\sigma'\in\pi\sb{14}\bS{7}} of order $8$, since it is not
a suspension (though \w[).]{2\sigma'=E\sigma''} Incidentally,
this shows that \w{[\iot{6},\iot{6}]\circ\eta\sb{11}=0} in
\w[.]{\pi\sb{12}\bS{6}}
\end{remark}

We conclude this section by showing

\begin{thm}\label{twhitpgen}
All elements in the homotopy groups of a finite wedge of
simply-connected spheres are generated of higher order by the
fundamental classes and their Whitehead products.
\end{thm}

\begin{proof}
Since any suspension is the value of a secondary homotopy operation (see \S \ref{rsusp}),
by Corollary \ref{catomic}, it suffices to show this for
\w[,]{\eta\sb{2}\in\pi\sb{3}\bS{2}} \w[,]{\nu\sb{4}\in\pi\sb{7}\bS{4}} and
\w[,]{\sigma\sb{8}\in\pi\sb{15}\bS{8}} and the maps
\w{\alpha\sb{1}(p)\in\pi\sb{2p}\bS{3}} for odd primes $p$.

Indeed, from \wref{eqwhco} we see that by Lemma \ref{lsusp}(a):
\begin{enumerate}
\renewcommand{\labelenumi}{(\alph{enumi})~}
\item \w{\eta\sb{2}} is represented by either half of
$\omega$ \wh e.g.,
\w{h:=s\sb{0}\iot{1}\cdot s\sb{1}\iot{1}
  \cdot s\sb{0}\iot{1}\cdot\sp{-1} s\sb{1}\iot{1}\sp{-1}} in\w{\pi\sb{1}\bW\sb{2}}
for \w[;]{\Wd=\bS{1}\otimes S\sp{1}}
\item \w{\nu\sb{4}} is represented by
  \w{[s\sb{0}i\sb{3},\,s\sb{1}i\sb{3}]\in\pi\sb{5}\bW\sb{2}} for
  \w[;]{\Wd=\bS{3}\otimes S\sp{1}}
\item \w{\sigma\sb{8}} is represented by
  \w{[s\sb{0}i\sb{7},\,s\sb{1}i\sb{7}]\in\pi\sb{13}\bW\sb{2}} for
  \w[.]{\Wd=\bS{7}\otimes S\sp{1}}
\end{enumerate}

Now for any odd prime $p$ we have
\begin{myeq}\label{eqhtpyoddp}
\pi\sb{i}\bS{2}\lo{p}~:=~\begin{cases}
\Zp\lra{\iot{2}} & \text{for}\ \ i=2\\
\Zp\lra{[\iot{2},\iot{2}]} & \text{for}\ \ i=3\\
0& \text{for} \ 4\leq i< 2p
\end{cases}
\end{myeq}
\noindent with \w[.]{[[\iot{2},\iot{2}],\iot{2}]=0}
Since \w{\alpha\sb{1}(p)\in\pi\sb{2p}\bS{3}} is not a suspension, it must appear in
filtration \w{k\geq 2} and in dimension \w{i=2p-k} in the spectral sequence for
\w[,]{\Wd=\bS{2}\lo{p}\otimes S\sp{1}} and all classes in \w{\pi\sb{i}\bW\sb{k}} are
iterated Whitehead products, by \wref{eqhtpyoddp} and Hilton's Theorem.
\end{proof}

\begin{example}\label{egalpha}
A Hall basis for the lowest dimension $2$-cycles for \w{\Wd=\bS{2}\lo{3}\otimes S\sp{1}}
is given by
\w{\vare\sb{1}:=[[s\sb{0}\iot{2},s\sb{1}\iot{2}],\,s\sb{0}\iot{2}]} and
\w{\vare\sb{2}:=[[s\sb{0}\iot{2},s\sb{1}\iot{2}],\,s\sb{1}\iot{2}]}
in \w[;]{\pi\sb{4}\bW\sb{2}} a Hall basis for the non-degenerate classes in
\w{\pi\sb{4}\bW\sb{3}} is given by \w{\theta\sb{1}:=
  [[s\sb{2}s\sb{0}\iot{2},s\sb{1}s\sb{0}\iot{2}],\,s\sb{2}s\sb{1}\iot{2}]}
and \w[.]{\theta\sb{2}:=
  [[s\sb{2}s\sb{1}\iot{2},s\sb{1}s\sb{0}\iot{2}],\,s\sb{2}s\sb{0}\iot{2}]}

We have \w{d\sb{0}\theta\sb{i}=0} and \w{d\sb{3}\theta\sb{i}=0} for \w[.]{i=1,2}
Since \w{[s\sb{1}\iot{2},s\sb{1}\iot{2}]=2(s\sb{1}\iot{2}\circ\eta\sb{2})}
and
\begin{myeq}\label{eqhtpymodt}
[s\sb{0}\iot{2},s\sb{1}\iot{2}\eta\sb{2}]~=~
[s\sb{0}\iot{2},s\sb{1}\iot{2}]\circ\eta\sb{3}~-~
[[s\sb{0}\iot{2},s\sb{1}\iot{2}],\,s\sb{1}\iot{2}]~\equiv~
-\vare\sb{2}\pmod{3}
\end{myeq}
\noindent by \cite[Theorem 3.24]{BHilJ}, we see that
\w[,]{d\sb{1}\theta\sb{1}=\vare\sb{2}} and \w[,]{d\sb{2}\theta\sb{1}=-2\vare\sb{2}}
so \w{[\vare\sb{2}]} has order $3$ in \w[.]{E\sp{2}\sb{2,4}=\pi\sb{2}\pi\sb{4}\Wd}
Similarly, \w[,]{d\sb{1}\theta\sb{2}=\vare\sb{2}} and
\w[,]{d\sb{2}\theta\sb{2}=\vare\sb{1}} so \w[.]{[\vare\sb{1}]=[\vare\sb{2}]}
Thus the generator \w{\alpha\sb{1}(3)} for \w{\pi\sb{6}\bS{3}\lo{3}\cong\ZZ/3}
is represented by either of these two classes.
\end{example}

%
%
\sect{Complex projective spaces}
\label{ccpn}

For each \w[,]{n\geq 1} the complex projective space \w{\CP{n}} fits into
a fibration sequence
\begin{myeq}\label{eqhopfcpn}
\bS{1}~\hra~\bS{2n+1}~\xra{g\sb{n}}~\CP{n}~,
\end{myeq}
\noindent with \w{g\sb{n}} a Hopf map, as well as a homotopy cofibration sequence
\begin{myeq}\label{eqcofcpn}
\bS{2n+1}~\xra{g\sb{n}}~\CP{n}~\xra{j\sb{n}}~\CP{n+1}~,
\end{myeq}
\noindent starting with \w{\CP{1}=\bS{2}} and \w[.]{g\sb{1}=\eta\sb{2}}
Thus \w{g\sb{n}} determines an isomorphism
\w{\pi\sb{i}\bS{2n+1}\cong\pi\sb{i}\CP{n}} for \w[,]{i\neq 2}
with \w{\pi\sb{2}\CP{n}\cong\ZZ} generated by
\w[.]{\iot{2}=j\sb{n-1}\circ\dotsc\circ j\sb{1}}

It is nevertheless illuminating to analyze how all elements in \w{\pis\CP{n}}
are generated by the unique indecomposable \w{\iot{2}} using (higher) homotopy
operations, as follows:

\begin{mysubsection}{The four dimensional complex projective space}
\label{scptwo}
Our approach is necessarily inductive, starting with
\w[.]{\CP{2}} From \wref{eqcofcpn} we see that a (minimal) simplicial resolution
\w{\Wd} for \w{\CP{2}} as in Section \ref{cssa} must start with
\w{\bW\sb{0}=\bS{2}} and \w[,]{\bW\sb{1}=\bS{3}\vee s\sb{0}\bS{2}} with
\w{d\sb{0}} given by the Hopf map \w{\eta\sb{2}=g\sb{1}:\bS{3}\to\CP{1}=\bS{2}}
on the summand \w[.]{\bS{3}}

As in \S \ref{swhite}, the lowest dimensional $1$-cycle in \w{\pi\sb{i}\Wd}
(for \w[)]{i=4} is given by the Whitehead product
\w[.]{\gam{2}=[\iot{3},s\sb{0}\iot{2}]} We see that indeed
\begin{myeq}\label{eqfirsttodacpt}
d\sb{0}(\gam{2})~=~[\eta\sb{2},\iot{2}]~=~
[\iot{2},\iot{2}]\circ\eta\sb{3}~=~(2\eta\sb{2})\circ\eta\sb{3}~=~0~,
\end{myeq}
\noindent and since this is a permanent cycle which cannot be hit by a differential (for
dimension reasons), it must represent \w{g\sb{2}:\bS{5}\to\CP{2}} (up to sign).

Note that \wref{eqfirsttodacpt} exhibits \w{g\sb{2}} as a canonical value of
a secondary operation in the sense of Proposition \ref{pfiltrationn} (though
not as a Toda bracket in the usual sense). Moreover, this also defines a map
$\widehat{g}$ from the simplicial space \w{\Vd:=\bS{4}\otimes S\sp{1}}
to the simplicial $1$-skeleton of \w{\Wd} described above
(and thus to \w{\Wd} itself), realized by \w[:]{g\sb{2}}
\myqdiag[\label{eqmapss}]{
  \dotsc & \bV\sb{2} && \bV\sb{1} & \bV\sb{0}\\
  \dotsc & s\sb{0}\bS{4}\vee s\sb{1}\bS{4}
    \ar@<2ex>[rr]\sp(0.5){d\sb{0}} \ar[rr]\sp(0.5){d\sb{1}}
    \ar@<-1ex>[rr]\sb(0.5){d\sb{2}}
    \ar[dd]\sb{\widehat{g}\sb{2}=}
    \sp{[s\sb{0}\iot{3},s\sb{1}s\sb{0}\iot{2}]\bot
      [s\sb{1}\iot{3},s\sb{1}s\sb{0}\iot{2}]}
    &&
    \bS{4} \ar[dd]\sb{\widehat{g}\sb{1}=}\sp{[\iot{3},s\sb{0}\iot{2}]}
 \ar@<0.5ex>[r]\sp(0.5){d\sb{0}}\ar@<-0.5ex>[r]\sb(0.5){d\sb{1}} &
 \ast \ar[dd]\sp{\widehat{g}\sb{0}} \\ \\
 \dotsc & s\sb{0}\bS{3}\vee s\sb{1}\bS{3}\vee s\sb{1}s\sb{0}\bS{2}
    \ar@<2ex>[rr]\sp(0.7){d\sb{0}} \ar[rr]\sp(0.7){d\sb{1}}
  \ar@<-1ex>[rr]\sb(0.7){d\sb{2}} && \bS{3}\vee s\sb{1}s\sb{0}\bS{2}
 \ar@<0.5ex>[r]\sp(0.7){d\sb{0}}\ar@<-0.5ex>[r]\sb(0.7){d\sb{1}} & \bS{2}\\
   \dotsc & \bW\sb{2} && \bW\sb{1} & \bW\sb{0}
}
\noindent From the description
in \S \ref{ssuspss} we see that for any map \w[,]{f:\bS{N}\to\bS{4}} precomposition
with \w{f\otimes S\sp{1}:\bS{N}\otimes S\sp{1}\to\Vd}
induces precomposition with \w{\Sigma f} in \w[.]{\pi\sb{5}\CP{2}} Thus
\w{[\gam{2}\circ f]\in\pi\sb{N}\bW\sb{1}} represents
\w[.]{[g\sb{2}\circ\Sigma f]\in\pi\sb{N+1}\CP{2}}
\end{mysubsection}

\begin{remark}\label{rcompcpt}
Any class \w{\alpha\in\pi\sb{k}\bS{5}\cong\pi\sb{k}\CP{2}} \wb{k>5} which is
not a suspension is represented by an element in higher filtration in the
spectral sequence for \w[.]{\Vd} The map of simplicial spaces
\w{\widehat{g}:\Vd\to\Wd} then induces a map of spectral sequences
\w[,]{\widehat{g}\sb{\ast}} yielding a representative for the
composite \w{g\sb{2}\circ\alpha} (which in this case is always non-trivial
since \w{(g\sb{2})\sb{\#}} is an isomorphism).
\end{remark}

\begin{example}\label{egcompcpt}
The first such class is
\w[,]{\sigma'''\in\pi\sb{12}\bS{5} \cong\ZZ/30} of order $2$. By \S \ref{rsigmappp},
this is represented by the element
\w{[s\sb{0}\iot{4},\,s\sb{1}\iot{4}]\circ\eta\sp{3}} in
\w[,]{E\sp{1}\sb{2,10}=\pi\sb{10}\bV\sb{2}} which maps under
\w{(\widehat{g}\sb{2})\sb{\ast}} to the iterated Whitehead product
\begin{equation*}
\begin{split}
&[[s\sb{0}\iot{3},s\sb{1}s\sb{0}\iot{2}],\
  [s\sb{1}\iot{3},s\sb{1}s\sb{0}\iot{2}]]~\circ\eta\sp{3}\\
  \quad&~~~~~~~~=~[[[s\sb{1}\iot{3},s\sb{1}s\sb{0}\iot{2}],\
      s\sb{1}s\sb{0}\iot{2}],\ s\sb{0}\iot{3}]~\circ\eta\sp{3}~+~
[[[s\sb{1}\iot{3},s\sb{1}s\sb{0}\iot{2}],\ s\sb{0}\iot{3}],\
  s\sb{1}s\sb{0}\iot{2}]~\circ\eta\sp{3}
\end{split}
\end{equation*}
\noindent in \w[.]{\pi\sb{10}\bW\sb{2}}
\end{example}

\begin{mysubsection}{Higher projective spaces}
\label{scphigh}
A simplicial resolution for \w{\CP{3}} will start with
\w{\bW\sb{0}=\bS{2}} and \w[,]{\bW\sb{1}=\bS{3}\vee s\sb{0}\bS{2}} as for
\w[,]{\CP{2}} but now we set
\w[,]{\bW\sb{2}:=\bS{4}\vee s\sb{0}\bS{3}\vee s\sb{1}\bS{3}\vee s\sb{1}s\sb{0}\bS{2}}
with \w{d\sb{0}} given by the representative for \w{g\sb{2}:\bS{5}\to\CP{2}} \wwh
that is, by \w{\gam{2}=[\iot{3},s\sb{0}\iot{2}]} \wwh
on the summand \w[.]{\bS{4}}

As before, the lowest dimensional $2$-cycle in \w{\pi\sb{i}\Wd} is given by
\begin{myeq}\label{eqgammathree}
\gam{3}~:=~[\iot{4},s\sb{1}s\sb{0}\iot{2}]~-~
         [s\sb{0}\iot{3},s\sb{1}\iot{3}]~\in~\pi\sb{5}\bW\sb{2}~.
\end{myeq}
\noindent By the Jacobi identity
$$
[[\iot{3},s\sb{0}\iot{2}],s\sb{0}\iot{2}]~+~
[[s\sb{0}\iot{2},s\sb{0}\iot{2}],\iot{3}]~+~
[[s\sb{0}\iot{2},\iot{3}],s\sb{0}\iot{2}]~=~0~,
$$
\noindent so
$$
2[[\iot{3},s\sb{0}\iot{2}],s\sb{0}\iot{2}]~=~
-[[s\sb{0}\iot{2},s\sb{0}\iot{2}],\iot{3}]~=~
-2[s\sb{0}\iot{2}\circ\eta\sb{2},\iot{3}]~=~
2[\iot{3},s\sb{0}\iot{2}\circ\eta\sb{2}]~.
$$
\noindent Therefore,
\begin{equation*}
d\sb{0}(\gam{3})~=~
[[\iot{3},s\sb{0}\iot{2}],s\sb{0}\iot{2}]-
[\iot{3},s\sb{0}\iot{2}\circ\eta\sb{2}]
~=~[\iot{3},s\sb{0}\iot{2}\circ\eta\sb{2}]
-[\iot{3},s\sb{0}\iot{2}\circ\eta\sb{2}]=0~,
\end{equation*}
\noindent Since \w[,]{d\sb{1}\gam{3}=[\iot{3},\iot{3}]=0} and
\w[,]{d\sb{2}\gam{3}=0} too, this is a permanent cycle. It cannot be hit
by a differential (for dimension reasons), so it must represent
\w{g\sb{3}:\bS{7}\to\CP{3}} as a third-order simplicial operation.

It also defines a map \w{\widehat{g}\sb{3}} from the simplicial space
\w{\Ud:=\bS{5}\otimes S\sp{2}} to \w[,]{\Wd} realized by \w[.]{g\sb{3}}
Once more, for any map \w[,]{f:\bS{N}\to\bS{5}} precomposition
with \w{f\otimes S\sp{2}:\bS{N}\otimes S\sp{2}\to\Ud}
shows that \w{[\gam{3}\circ f]\in\pi\sb{N}\bW\sb{2}} represents
\w[.]{[g\sb{3}\circ\Sigma\sp{2}f]\in\pi\sb{N+2}\CP{3}}

Any class in \w{\pi\sb{k}\bS{7}\cong\pi\sb{k}\CP{3}} \wb{k>7} which is
not a double suspension is represented by an element in higher filtration in the
spectral sequence for \w[.]{\Vd} The first such class is
\w{\sigma'\in\pi\sb{14}\bS{7} \cong \ZZ / 120} \wwh in fact, \w{\sigma'} is not a suspension,
and \w{2\sigma'=E\sigma''} is not a double suspension, but
\w[.]{4\sigma'=E\sp{2}\sigma'''}
\end{mysubsection}

\begin{mysubsect}{Some combinatorics}
\label{scomb}

In order to describe the general case, we require some combinatorial notions:
\end{mysubsect}

\begin{defn}\label{dsign}
Given a $k$-tuple \w{I=(i\sb{1}<\dotsc<i\sb{k})} of non-negative integers (in ascending
order), let $\uI$ denote the underlying unordered set (and conversely),
and \w{s\sb{I}=s\sb{i\sb{k}}s\sb{i\sb{k-1}}\dotsc s\sb{i\sb{2}}s\sb{i\sb{1}}} the
corresponding iterated degeneracy map, so that
\begin{myeq}\label{eqdegens}
  \forall ~ x~\exists ~y\ s\sb{I}x=s\sb{j}y\hs\text{if and only if}\hs j\in I~.
\end{myeq}

If $\uI$ and $\uJ$ are disjoint sets of natural numbers, \w{I\sqcup J=J\sqcup I}
will denote the disjoint union \w[,]{\uI\sqcup\uJ} in ascending order.

Given two finite sets $I$ and $J$ of non-negative integers, each arranged in
ascending order, their \emph{sign} \w{\sgn{I,J}} is defined as follows:

\begin{enumerate}
\renewcommand{\labelenumi}{(\roman{enumi})~}
\item If $\uI$ and $\uJ$ form a partition of \w[,]{\{1,2,\dotsc,n\}} we write
\w{\sigma\sb{(I,J)}\in S\sb{n}} for the permutation obtained by concatenating
$I$ with $J$, and let \w[.]{\sgn{I,J}:=\sgn{\sigma\sb{(I,J)}}}

\item If $\uI$ and $\uJ$ are disjoint subsets of \w[,]{\{0,1,2,\dotsc,N\}} let
\w{\pi:I\sqcup J\to\{1,2,\dotsc,n\}} be an order-preserving
isomorphism, with \w[.]{\sgn{I,J}:=\sgn{\pi[I],\pi[J]}}
\item If $\uI$ and $\uJ$ are any two finite sets $I$ and $J$ of non-negative integers,
  let \w{I':=\uI\setminus\uJ} and \w{J':=\uJ\setminus\uI} (each in ascending order),
  and set \w[.]{\sgn{I,J}:=\sgn{I',J'}}
\end{enumerate}
\end{defn}

\begin{example}\label{egsign}
For \w{\uI=\{2,4\}} and \w[,]{\uJ=\{1,3,5\}} we have
\w[,]{\sigma\sb{(I,J)}=\binom{1\ 2\ 3\ 4\ 5}{2\ 4\ 1\ 3\ 5}} so \w[.]{\sgn{I,J}=-1}

For \w{\uI=\{1,5\}} and \w[,]{\uJ=\{0,2,7\}} the map $\pi$ is given
by \w{\binom{0\ 1\ 2\ 5\ 7}{1\ 2\ 3\ 4\ 5}} and so
\w{\pi[I]=\{2,4\}} and \w[,]{\pi[J]=\{1,3,5\}} and thus \w[.]{\sgn{I,J}=-1}

For \w{\uI=\{1,3,5,6\}} and \w[,]{\uJ=\{0,2,3,6,7\}} we have
\w{I'=\{1,5\}} and \w[,]{J'=\{0,2,7\}} so again \w[.]{\sgn{I,J}=-1}
\end{example}

\begin{notn}\label{ncpns}
Given \w[,]{n\geq 2} for each \w[,]{0\leq k<n-1} let \w{\cII\sb{k}\sp{n}}
denote the collection of all \wwb{k,n-k-1}partitions \w{(I,J)} of \w[.]{\{0,\dotsc,n-2\}}
When \w[,]{2k=n-1} we include in \w{\cII\sb{k}\sp{n}} only sets with
\w{0\in I} (to avoid double counting).
\end{notn}

\begin{remark}\label{rcpns}
If \w{(I',J')} is obtained from \w{(I,J)} by switching a pair of
elements between $I$ and $J$ (while maintaining ascending order),
then \w[.]{\sgn{I',J'}=-\sgn{I,J}} As a result, if
\w{\uI\setminus\uJ} has cardinality $k$ and \w{\uJ\setminus\uI}
has cardinality $\ell$, then
\begin{myeq}\label{eqsgncompl}
\sgn{I,J}~=~(-1)\sp{k\cdot\ell}\sgn{J,I}.
\end{myeq}

Finally, if \w{M=(m\sb{1},\dotsc,m\sb{k})} is an ordered $k$-tuple of natural numbers (in
ascending order), denote its underlying set by $\uM$. Conversely,
For any decomposition \w[,]{\uM\sqcup\uN\sqcup\uP=\{0,\dotsc,n-1\}} the
sign of corresponding three-fold shuffle \w{(M,N,P)} satisfies:
\begin{myeq}\label{eqsgndisunl}
\sgn{M,N\sqcup P}\cdot\sgn{N,P}~=~\sgn{M,N,P}~=~\sgn{M\sqcup N,P}\cdot\sgn{M,N}~.
\end{myeq}
\end{remark}

We can now state our main result for this section:

\begin{prop}\label{pcpn}
For each \w[,]{n\geq 1} there is a rational sequential approximation \w{\Wd} for
\w{\CP{n}} with a single non-degenerate \wwb{k+2}sphere in \w{\bW\sb{k}} for each
\w[.]{0\leq k<n} The Hopf map \w{g\sb{n}:\bS{2n+1}\to\CP{n}} is represented
in the homotopy spectral sequence for \w{\Wd} by the \wwb{n-1}cycle
\begin{myeq}\label{eqgamman}
\gam{n}~:=~
\sum\sb{j=2}\sp{\lfloor\frac{n+3}{2}\rfloor}\ \sum\sb{(I,J)\in\cII\sb{j-2}\sp{n}}\
(-1)\sp{n\cdot j}
\sgn{I,J}\cdot[s\sb{I}\iot{n-j+3},\ s\sb{J}\iot{j}]
\end{myeq}
\noindent in \w[.]{\pi\sb{n+2}\bW\sb{n-1}} This also serves as the face map
\w{d\sb{0}} on \w{\iot{n+1}\in\pi\sb{n+1}\bS{n+1}}
in simplicial dimension \w{n-1} in the sequential approximation for \w[.]{\CP{n+1}}
\end{prop}

We shall refer to \w{(-1)\sp{n\cdot j}} as the \emph{global coefficient} of
\w[.]{[s\sb{I}\iot{n-j+3},\ s\sb{J}\iot{j}]}

\begin{proof}
All iterated Whitehead products in \w{\pi\sb{i}\bW\sb{j}} are degenerate for
\w[,]{i\leq j>n} and \w{\gam{n}} is the first non-trivial class that
does not come from \w[.]{\CP{n-1}} By Proposition \ref{pnoratho}, it thus suffices
to show that \w{\gam{n}} is a Moore cycle in the integral homotopy spectral sequence
for \w[\vsm.]{\Wd}

\noindent\textbf{Step 1:}\hs If \w{J:=\Phi(I)} and \w[,]{r\geq 1} then
\w{d\sb{r}([s\sb{I}\iot{n-j+3},s\sb{J}\iot{j}]=0} unless one of \w{\{r,r+1\}} is in
$I$ and one is in $J$. If \w{(I',J')} is then obtained from \w{(I,J)} by
interchanging $r$ and \w[,]{r+1}  \w{\sgn{I',J'}=-\sgn{I,J}} by Remark \ref{rcpns},
while \w[.]{d\sb{r}[s\sb{I}\iot{n-j+3},s\sb{J}\iot{j}]=
  d\sb{r}[s\sb{I'}\iot{n-j+3},s\sb{J'}\iot{j}]}

By our assumption in \S \ref{ncpns}, it remains to deal with
\w[,]{d\sb{1}([s\sb{I}\iot{k+2},s\sb{J}\iot{k+2}])} for \w{n=2k+1} and \w{|I|=|J|=k}
with \w[.]{0\in I} So
\w{I=(0=i\sb{1}<i\sb{2}\dotsc<i\sb{k})} and \w{J=(1=j\sb{1}<j\sb{2}\dotsc<j\sb{k})}
(the case \w{k=1} is dealt with in \wref[\textit{ff.}).]{eqgammathree}
If \w{(I',J')} is obtained from \w{(I,J)} by interchanging \w{i\sb{i}} with \w{j\sb{i}}
for all \w{2\leq i\leq k} (leaving \w{0\in I'} and \w[),]{1\in J'}
then \w[,]{\sgn{I,J}=(-1)\sp{(k-1)(k-1)}\sgn{I',J'}} while
$$
d\sb{1}([s\sb{I}\iot{k+2},s\sb{J}\iot{k+2}])~=~
(-1)\sp{(k+2)(k+2)}\cdot d\sb{1}([s\sb{I'}\iot{k+2},s\sb{J'}\iot{k+2}])
$$
\noindent by \cite[(7.5)]{GWhE}, so the two appear in \w{d\sb{1}\gam{n}} with
opposite signs. Thus we see that \w{d\sb{r}\gam{n}=0} for all \w[\vsm .]{1\leq r\leq n}

\noindent\textbf{Step 2:}\hs By the Jacobi identity for any \w{n\geq 2}
and iterated degeneracy \w{s\sb{I}} we have
\w[,]{[[\iota\sb{n},s\sb{I}\iota\sb{2}],s\sb{I}\iota\sb{2}]+
[[s\sb{I}\iota\sb{2},s\sb{I}\iota\sb{2}],\iota\sb{n}]+
[[s\sb{I}\iota\sb{2},\iota\sb{n}],s\sb{I}\iota\sb{2}]=0} so
\begin{equation*}
2[[\iota\sb{n},s\sb{I}\iota\sb{2}],s\sb{I}\iota\sb{2}]~=~
-[s\sb{I}[\iota\sb{2},\iota\sb{2}],\iota\sb{n}]~=~
-2[s\sb{I}\eta\sb{2}\sp{\#}\iota\sb{2},\iota\sb{n}]~=~
(-1)\sp{n+1}2[\iota\sb{n},s\sb{I}\eta\sb{2}\sp{\#}\iota\sb{2}]
\end{equation*}
\noindent and thus
\begin{myeq}\label{eqhopfwh}
[[\iota\sb{n},s\sb{I}\iota\sb{2}],s\sb{I}\iota\sb{2}]~=~
  (-1)\sp{n+1}[\iota\sb{n},s\sb{I}\eta\sb{2}\sp{\#}\iota\sb{2}]
\end{myeq}
\noindent (since all summands in Hilton's Theorem in this dimension are
infinite cyclic). Thus
\begin{myeq}\label{eqthreewh}
d\sb{0}([s\sb{0}\iot{n},s\sb{J}\iot{3}])=[\iot{n},s\sb{J'}\gam{1}]
=[\iot{n},s\sb{J'}\eta\sb{2}\sp{\#}\iot{2}]
\end{myeq}
\noindent (where \w{J=\{1,\dotsc,n-2)} and \w[)]{J'=\{0,\dotsc,n-3\}} \wwh
the only summand of \w{d\sb{0}\gam{n}} not a three-fold Whitehead product \wh
equals the first summand of
\w{d\sb{0}([\iot{n+1},s\sb{J''}\iot{2}])=[\gam{n-1},s\sb{J'}\iot{2}]}
for \w{J''=J'\cup\{n-2\}} \wwh i.e.,
\w[,]{[\iot{n},s\sb{J'}\iot{2}],s\sb{J'}\iot{2}]} with opposite sign
(due to the sign in \wref{eqhopfwh} and the coefficient \w{(-1)\sp{n\cdot j}}
in \wref[,]{eqgamman} where \w{j=3} in \wref[)\vsm.]{eqthreewh}

\noindent\textbf{Step 3:}\hs To show that \w{\gam{n}} is a Moore cycle, consider
triple Whitehead products summands of \w{d\sb{0}\gam{n}} of the form
\begin{myeq}\label{eqA}
A~=~[[s\sb{I}\iot{p},\,s\sb{J}\iot{q}],\,s\sb{K}\iot{r}]]~,
\end{myeq}
\noindent associated to a partition \w[,]{\{0,1,\dotsc,n-2\}=L\sqcup M\sqcup N}
with cardinalities \w[,]{|N|=q-2} \w[,]{|M|=p-2} and \w{|L|=r-2} (so \w[),]{p+q+r=n+4}
such that \w[,]{I=L\sqcup N} \w[,]{J=L\sqcup M} and \w{K=M\sqcup N}
(and thus \w[,]{|I|=n-p} \w[,]{|J|=n-q} and \w[).]{|K|=n-r} By \wref{eqgamman} and \S
\ref{ncpns} we must assume \w{p+q-1\geq r} and \w[,]{p\geq q} and if \w[,]{p=q}
then \w[.]{\min(N)<\min(M)}
Given $A$, set
\begin{myeq}\label{eqBC}
B~=~[[s\sb{J}\iot{q},s\sb{K}\iot{r}],s\sb{I}\iot{p}]\hs\text{and}\hs
C~=~[[s\sb{K}\iot{r},s\sb{I}\iot{p}],s\sb{J}\iot{q}]~,
\end{myeq}
\noindent with
\begin{myeq}\label{eqjacobiabc}
  (-1)\sp{pr}A~+~(-1)\sp{qp}B~+~(-1)\sp{rq}C~=~0
\end{myeq}
\noindent by the Jacobi identity (see \cite[X, (7.14)]{GWhE}).
Thus to show that \w[,]{d\sb{0}\gam{n}=0} we must show that each of
$A$, $B$, and $C$ appears exactly once in the expansion of \w[,]{d\sb{0}\gam{n}}
with the appropriate sign.

By \wref[,]{eqgamman} $A$ only appears in
\w[,]{d\sb{0}([s\sb{\hL}\iot{p+q-1},s\sb{\hhK}\iot{r}])=
  [s\sb{L}d\sb{0}\gam{p+q-3},s\sb{K}\iot{r}]}
corresponding to \w{[s\sb{N}\iot{p},\,s\sb{M}\iot{q}]}
in \w[,]{\gam{p+q-3}} with $\hL$ obtained from $L$ by adding one to each index, and
$\hhK$ obtained from $K$ by adding one to each index and adjoining $0$, so
\begin{myeq}\label{eqhlhk}
  \sgn{\hhK,\hL}~=~\sgn{K,L}\hs\text{and}\hs
  \sgn{\hL,\hhK}~=~(-1)\sp{r}\sgn{L,K}~.
\end{myeq}

By \wref[,]{eqdegens} \w{s\sb{I}=s\sb{L}\circ s\sb{N}}
and \w[,]{s\sb{J}=s\sb{L}\circ s\sb{M}} and by Definition \ref{dsign}(ii)-(iii):
\begin{myeq}\label{eqlmna}
\sgn{I,J}=\sgn{N,M},\  \sgn{J,K}=\sgn{L,N},\ \sgn{I,K}=\sgn{L,M}.
\end{myeq}
\noindent Finally, by \wref{eqsgndisunl} we have
\begin{myeq}\label{eqlmnb}
\begin{split}
  \sgn{L,K}~=&~\sgn{L,M\sqcup N}~=~\sgn{L,M}\cdot\sgn{M,N}\cdot\sgn{J,N}\\
  =&~\sgn{L,N\sqcup M}~=~\sgn{L,N}\cdot\sgn{N,M}\cdot\sgn{I,M}~.
\end{split}
\end{myeq}

\noindent\textbf{Step 4:}\hs We see that by \wref[,]{eqhlhk} the coefficient of
\w{[s\sb{\hL}\iot{p+q-1},s\sb{\hK}\iot{r}]} in \w{\gam{n}} is
\begin{equation*}
(-1)\sp{nr}\sgn{\hL,\hhK}=(-1)\sp{nr+r}\sgn{L,K}=
  (-1)\sp{(p+q)r}\sgn{L,M}\sgn{M,N}\sgn{J,N}
\end{equation*}
\noindent while the sign of \w{[s\sb{N}\iot{p},\,s\sb{M}\iot{q}]} in \w{\gam{p+q-3}} is
\begin{equation*}
(-1)\sp{(p+q-3)q}\sgn{N,M}=(-1)\sp{pq}\sgn{N,M}=\sgn{M,N}~.
\end{equation*}
\noindent Therefore, if we set
\begin{myeq}\label{eqtheta}
\Theta~:=~(-1)\sp{qr}\sgn{L,M}\cdot\sgn{J,N}~,
\end{myeq}
\noindent the sign of $A$ in \w{d\sb{0}(\gam{n})} is
\w[,]{(-1)\sp{pr}\cdot\Theta} so multiplication by \w{(-1)\sp{pr}} as in
\wref{eqjacobiabc} yields $\Theta$ as the ``corrected'' coefficient of $A$\vsm.

\noindent\textbf{Step 5:}\hs We now assume that \w{q>r} and \w{q+r-1>p}
(or \w{q+r-1=p} and \w[).]{0\not\in I}
We see that $B$ of \wref{eqBC} appears as a summand in
\w{d\sb{0}(\wB)=[s\sb{M}d\sb{0}(\iot{q+r-1},s\sb{I}\iot{p}])}
for the summand \w{\wB:=[s\sb{\hM}\iot{q+r-1},s\sb{\hhI}\iot{p}]} in
\wref[,]{eqgamman} which has coefficient
\begin{equation*}
\begin{split}
&(-1)\sp{np}\sgn{\hM,\hhI}~=~(-1)\sp{np+p}\sgn{M,I}~=~
  (-1)\sp{(q+r)p}\sgn{M,L\sqcup N}\\
  &=(-1)\sp{(q+r)p}\sgn{M,L}\sgn{L,N}\sgn{J,N}=
  (-1)\sp{pq}\sgn{L,M}\sgn{L,N}\sgn{J,N}
\end{split}
\end{equation*}
\noindent in \w{\gam{n}} by \wref{eqhlhk} and \wref[.]{eqsgndisunl}

On the other hand, the summand \w{[s\sb{N}\iot{q},s\sb{L}\iot{r}]}
has coefficient \w{(-1)\sp{(q+r-3)r}\sgn{L,N}} in \w[,]{\gam{q+r-3}} so altogether
the coefficient of $B$ in \w{d\sb{0}(\gam{n})} is
$$
(-1)\sp{pq+qr}\sgn{L,M}\sgn{J,N}~=~(-1)\sp{pq}\cdot\Theta~,
$$
\noindent so multiplication by \w{(-1)\sp{qp}} as in \wref{eqjacobiabc} yields
$\Theta$ as the corrected coefficient of $B$.

Note that if \w[,]{q+r-1<p} instead of $B$ we would have
\w{B':=[s\sb{I}\iot{p},[s\sb{J}\iot{q},s\sb{K}\iot{r}]]}
in \w{d\sb{0}(\wB')=[s\sb{I}\iot{p},s\sb{M}d\sb{0}(\iot{q+r-1}])}
for the summand \w{\wB':=[s\sb{\hhI}\iot{p},s\sb{\hM}\iot{q+r-1}]} of
\wref[.]{eqgamman} By \cite[(7.5)]{GWhE} \w[,]{B'=(-1)\sp{p(q+r-1)}B} while
the global coefficient \w{(-1)\sp{n(q+r-1)}} of \w{\wB'} agrees with
\w{(-1)\sp{np}} for $\wB$, since \w[.]{(-1)\sp{n(n-1))}=+1}
On the other hand, by \wref{eqsgncompl}
$$
\sgn{\hhI,\hM,}~=~(-1)\sp{(p-2)(q+r-5)}\sgn{\hM,\hhI}~=~(-1)\sp{p(q+r-1)}\sgn{\hM,\hhI}~,
$$
\noindent so if we replace \w{B'} by $B$ the total sign is unchanged\vsm.

\noindent\textbf{Step 6:}\hs Since \w[,]{r<p} $C$ of \wref{eqBC} must be
replaced by \w{C':=[[s\sb{I}\iot{p},s\sb{K}\iot{r}],s\sb{J}\iot{q}]}
(assuming \w[),]{p+r-1\geq q} with \w{C'=(-1)\sp{pr}C} by \cite[(7.5)]{GWhE}.

As above, \w{C'} appears as a summand in
\w{d\sb{0}(\wC')=[s\sb{N}d\sb{0}(\iot{p+r-1}),s\sb{J}\iot{q}])}
for \w{\wC':=[s\sb{\hN}\iot{p+r-1},s\sb{\hhJ}\iot{q}]} in
\wref[,]{eqgamman} with coefficient
\begin{equation*}
  (-1)\sp{nq}\sgn{\hN,\hhJ}~=~(-1)\sp{nq}(-1)\sp{(n-q+1)q}\sgn{\hhJ,\hN}~=~\sgn{J,N}
\end{equation*}
\noindent in \w[,]{\gam{n}} by \wref{eqhlhk} and \wref[.]{eqsgncompl}

On the other hand, the summand \w{[s\sb{I}\iot{p},s\sb{K}\iot{r}]} has coefficient
$$
(-1)\sp{(p+r-3)p}\sgn{I,K}~=~(-1)\sp{pr}\sgn{L,M}
$$
\noindent in \w[,]{\gam{q+r-3}} by \wref[.]{eqlmna}
Thus the coefficient of \w{C'} in \w{d\sb{0}(\gam{n})} is
\w[,]{(-1)\sp{(p+q)r}\cdot\Theta} which means that of $C$ is
\w[,]{(-1)\sp{qr}\cdot\Theta}
and thus multiplication by \w{(-1)\sp{qr}} again yields $\Theta$
as the corrected coefficient of $C$.
This shows that \wref{eqjacobiabc} indeed holds, so these three terms in the
boundary sum to $0$. Thus \w{\gam{n}} is a Moore cycle, as required.
\end{proof}

\begin{example}\label{eggammas}
The next cycles after \w{\gam{3}} of \wref{eqgammathree} are
$$
\gam{4}=[\iot{5},s\sb{2}s\sb{1}s\sb{0}\iot{2}]+[s\sb{0}\iot{4},s\sb{2}s\sb{1}\iot{3}]
-[s\sb{1}\iot{4},s\sb{2}s\sb{0}\iot{3}] +[s\sb{2}\iot{4},s\sb{1}s\sb{0}\iot{3}]
$$
in \w[,]{\pi\sb{6}\bW\sb{3}} with
\begin{equation*}
\begin{split}
d\sb{0}(\gam{4})=&[[\iot{4},s\sb{1}s\sb{0}\iot{2}],s\sb{1}s\sb{0}\iot{2}]
-[[s\sb{0}\iot{3},s\sb{1}\iot{3}],s\sb{1}s\sb{0}\iot{2}]
+[\iot{4},s\sb{1}s\sb{0}\eta\sb{2}\sp{\#}\iot{2}]\\
&-[[s\sb{0}\iot{3},s\sb{1}s\sb{0}\iot{2}],s\sb{1}\iot{3}]
+[[s\sb{1}\iot{3},s\sb{1}s\sb{0}\iot{2}],s\sb{0}\iot{3}]~,
\end{split}
\end{equation*}
\noindent which vanishes by combining \wref{eqhopfwh} with \wref[,]{eqjacobiabc} and
\begin{equation*}
\begin{split}
\gam{5}=&~[\iot{6},s\sb{3}s\sb{2}s\sb{1}s\sb{0}\iot{2}]
-[s\sb{0}\iot{5},s\sb{3}s\sb{2}s\sb{1}\iot{3}]
+[s\sb{1}\iot{5},s\sb{3}s\sb{2}s\sb{0}\iot{3}]
-[s\sb{2}\iot{5},s\sb{3}s\sb{1}s\sb{0}\iot{3}]\\
&+[s\sb{3}\iot{5},s\sb{2}s\sb{1}s\sb{0}\iot{3}]
+[s\sb{1}s\sb{0}\iot{4},s\sb{3}s\sb{2}\iot{4}]
-[s\sb{2}s\sb{0}\iot{4},s\sb{3}s\sb{1}\iot{4}]
+[s\sb{3}s\sb{0}\iot{4},s\sb{2}s\sb{1}\iot{4}]~.
\end{split}
\end{equation*}
\end{example}

\begin{remark}\label{rgammas}
By Proposition \ref{pnoratho} implies that one can construct an integral simplicial
resolution \w{\Wd} of \w{\CP{n}} such that any element
\w{[\alpha]\in E\sp{2}\sb{r,s}} of infinite order in the spectral sequence
has a multiple which is a permanent cycle. We hope to show in \cite{BBSenA} that
the description in Proposition \ref{pcpn} is in fact valid integrally.
\end{remark}


\begin{thebibliography}{DKSm}
%
\bibitem[Ada]{AdHI}
J.F.~Adams,
``On the non-existence of elements of Hopf invariant one'',\hsm
\textit{Ann.\ Math.\ (2)} \textbf{72} (1960), pp.~20-104.
%
\bibitem[Ade]{AdemI}
J.~Adem,
``The iteration of the {Steenrod} squares in algebraic topology'',\hsm
\textit{Proc.\ Nat.\ Acad.\ Sci.\ USA} \textbf{38} (1952), pp.\ 720-726.
%
\bibitem[AA]{AArkS}
P.G.~Andrews \& M.~Arkowitz,
``Sullivan's minimal models and higher order {Whitehead} products'',\hsm
\textit{Can.\ J.\ Math.} \textbf{30} (1978), pp.~961-982.
%
\bibitem[BSS1]{BSStapC}
T.~Barthel, T.M.~Schlank, \& N.J.~Stapleton,
``Chromatic homotopy theory is asymptotically algebraic'',\hsm
\textit{Inv.\ Math.} \textbf{220} (2020), 737-845.
%
\bibitem[BSS2]{BSStapM}
T.~Barthel, T.M.~Schlank, \& N.J.~Stapleton,
``Monochromatic homotopy theory is asymptotically algebraic'',\hsm
\textit{Adv.\ Math.} \textbf{393} (2021), Paper No. 107999.
%
\bibitem[BBS1]{BBSenT}
S.~Basu, D.~Blanc, \& D.~Sen,
``A note on Toda brackets'',\hsm
\textit{J.\ Homotopy \& Rel.\ Struct.} \textbf{15} (2020), pp.\ 495-510
%
\bibitem[BBS2]{BBSenH}
S.~Basu, D.~Blanc, \& D.~Sen,
``Higher structures in the unstable Adams spectral sequence'',\hsm
\textit{Homotopy, Homology, \& Appl.} \textbf{23} (2021), pp.\ 69-94.
%
\bibitem[BBS3]{BBSenA}
S.~Basu, D.~Blanc, \& D.~Sen,
``The algebra of higher homotopy operations'',\hsm
preprint, 2022.
%
\bibitem[BH]{BHilJ}
M.G.~Barratt \& P.J.~Hilton,
``On join operations in homotopy groups'',\hsm
\textit{Proc.\ Lond.\ Math.\ Soc.\ (3)} \textbf{3} (1953), pp.\ 430-445.
%
\bibitem[BBG]{BBGondH}
H.-J.~Baues, D.~Blanc, \& S.~Gondhali,
``Higher Toda brackets and Massey products'',\hsm
\textit{J.~Homotopy \& Rel.~Struct.} \textbf{11} (2016), pp.\ 643-677.
%
\bibitem[BJ]{BJiblSL}
H.-J.~Baues \& M.A.~Jibladze,
``Suspension and loop objects in theories and coho\-mo\-logy'',\hsm
\textit{Georgian Math.\ J.} \textbf{8} (2001), pp.~697-712.
%
\bibitem[B]{BergM}
J.E.~Bergner,
``A model cate\-gory struc\-ture on the ca\-tegory of simp\-licial cate\-go\-ries'',\hsm
\textit{Trans.\ AMS} \textbf{359} (2007), pp.~2043-2058.
\bibitem[BJT1]{BJTurnH}
D.~Blanc, M.W.~Johnson, \& J.M.~Turner,
``Higher homotopy operations and cohomology'',\hsm
\textit{J.\ $K$-Theory} \textbf{5} (2010), pp.~167-200.
%
\bibitem[BJT2]{BJTurnHA}
D.~Blanc, M.W.~Johnson, \& J.M.~Turner,
``A Constructive Approach to Higher Homotopy Operations'',\hsm
in D.G.~Davis, H.-W.~Henn, J.F.~Jardine, M.W.~Johnson, \& C.~Rezk, eds.,
\textit{Homotopy Theory: Tools and Applications}
  Contemp.\ Math.\ \textbf{729}, AMS, Providence, RI, 2019, pp.~21-74.
%
\bibitem[BJT3]{BJTurnHI}
D.~Blanc, M.W.~Johnson, \& J.M.~Turner,
``Higher invariants for spaces and maps',\hsm
\textit{Alg.\ Geom.\ Top.} \textbf{21} (2021), pp.\ 2425-2488.
%
\bibitem[BMa]{BMarkH}
D.~Blanc \& M.~Markl,
``Higher homotopy operations'',\hsm
\noindent{Math.\ Zeit.} \textbf{345} (2003), pp.~1-29.
%
\bibitem[BMe]{BMeadS}
D.~Blanc \& N.J.~Meadows,
``Spectral sequences in {$(\infty, 1)$}-categories'',\hsm
\textit{J.\ Pure Appl.\ Alg.} \textbf{226} (2022), Paper No. 106905.
%
\bibitem[BV]{BVoHI}
J.M.~Boardman \& R.M.~Vogt,
\textit{Homotopy Invariant Algebraic Structures on Topological  Spaces},\hsm
Springer-\-Verlag \textit{Lec.\ Notes Math.} \textbf{347}, Berlin-\-New York,
1973.
%
\bibitem[BF]{BFrieH}
A.K.~Bousfield \& E.M.~Friedlander,
``Homotopy theory of $\Gamma$-spaces, spectra, and bisimplicial sets'',\hsm
in M.G.~Barratt \& M.E.~Mahowald, eds.,
\textit{Geometric Applications of Homotopy Theory, II}
Springer-\-Verlag \textit{Lec.\ Notes Math.} \textbf{658},
Berlin-\-New York, 1978, pp.~80-130
%
\bibitem[BK]{BKanH}
A.K.~Bousfield \& D.M.~Kan,
\textit{Homotopy Limits, Completions, and Lo\-ca\-li\-za\-tions},\hsm
Sprin\-ger \textit{Lec.\ Notes Math.} \textbf{304}, Berlin-\-New York, 1972.
%
\bibitem[CF]{CFranH}
J.D.~Christensen \& M.~Frankland,
``Higher Toda brackets and the Adams spectral sequence in triangulated categories'',\hsm
\textit{Alg.\ Geom.\ Top.} \textbf{17} (2017), pp.\ 2687-2735.
%
\bibitem[Co]{JCohDS}
J.M.~Cohen,
``The decomposition of stable homotopy'',\hsm
\textit{Ann.\ Math.\ (2)} \textbf{87} (1968), pp.\ 305-320.
%
\bibitem[Cu]{CurtS}
  E.B.~Curtis,
  ``Simplicial homotopy theory'',\hsm
  \textit{Adv.\ Math.} \textbf{2} (1971), pp.\ 107-209.
%
\bibitem[DHKS]{DHKSmitM}
W.G.~Dwyer, P.S.~Hirschhorn, D.M.~Kan, \& J.H.~Smith,
\textit{Homotopy limit functors on model categories and homotopical categories},\hsm
Math.\ Surveys \& Monographs \textbf{113}, AMS, Providence, RI, 2004.
%
\bibitem[DK]{DKanS}
W.G.~Dwyer \& D.M.~Kan,
``Simplicial localizations of categories'',\hsm
\textit{J.\ Pure Appl.\ Alg.} \textbf{17} (1980), pp.~267-284.
%
\bibitem[DKSm]{DKSmH}
W.G.~Dwyer, D.M.~Kan, \& J.H.~Smith,
``Homotopy commutative diagrams and their realizations'',\hsm
\textit{J.\ Pure Appl.\ Alg.} \textbf{57} (1989), pp.\ 5-24.
%
\bibitem[DKSt]{DKStB}
W.G.~Dwyer, D.M.~Kan, \& C.R.~Stover,
``The bigraded homotopy groups $\pi\sb{i,j}X$ of a pointed simplicial space'',\hsm
\textit{J.\ Pure Appl.\ Alg.} \textbf{103} (1995), pp.~167-188.
%
\bibitem[GJ]{GJarS}
P.G.~Goerss \& J.F.~Jardine,
\textit{Simplicial Homotopy Theory},\hsm
Progress in Mathematics \textbf{179}, Birkh\"{a}user, Basel-Boston, 1999.
%
\bibitem[Hil]{HilH}
P.J.\ Hilton,
``On the homotopy groups of the union of spheres'', \hsm \
\textit{J.\ Lond.\ Math.\ Soc.} \textbf{30} (1955), pp.\ 154-172.
%
\bibitem[HW]{HWhiW}
P.J.\ Hilton \& J.H.C.~Whitehead,
``Note on the Whitehead product'',\hsm
\textit{Ann.\ Math.} \textbf{58} (1953), pp.\ 429-442.
%
\bibitem[Hir]{PHirM}
P.S.~Hirschhorn,
\textit{Model Categories and their Localizations},\hsm
Math.\ Surveys \& Monographs \textbf{99}, AMS, Providence, RI, 2002.
%
\bibitem[K]{Koc}
S.O.~Kochman,
\textit{Bordism, stable homotopy and Adams spectral sequences}, \hsm
Fields Institute Monographs \textbf{7}, AMS, Providence, RI, 1996.
%
\bibitem[Lin]{TLinHD}
``Homological dimensions of stable homotopy modules and their geometric
characterizations'',\hsm
\textit{Trans.\ AMS} \textbf{172} (1972), pp.~473-490.
%
\bibitem[Liu]{LiuFC}
A.~Liulevicius,
\textit{The factorization of cyclic reduced powers by secondary cohomology
operations},\hsm Mem.\ AMS \textbf{42}, AMS, Providence, RI, 1962.
%
\bibitem[Lu]{LuriH}
J.~Lurie,
\textit{Higher Topos Theory},\hsm
Ann.\ Math.\ Studies \textbf{170}, Princeton U.\ Press, Princeton, 2009.
%
\bibitem[Mas]{MassN}
W.S.~Massey,
``A new cohomology invariant of topological spaces'',\hsm
\textit{Bull.\ AMS} \textbf{57} (1951), p.\ 74.
%
\bibitem[May]{MayS}
J.P.~May,
\textit{Simplicial Objects in Algebraic Topology},\hsm
U.\ Chicago Press, Chicago, 1967.
%
\bibitem[N]{NisN}
G.~Nishida,
``The nilpotence of elements of the stable homotopy groups of spheres",
\textit{J.\ Math.\ Soc.\ Japan} \textbf{25} (1973), pp.~707-732.
%
\bibitem[Q1]{QuiS}
D.G.~Quillen,
``Spectral sequences of a double semi-\-simplicial group'',\hsm
\textit{Topology} \textbf{5} (1966), pp.~155-156.
%
\bibitem[Q2]{QuiR}
D.G.~Quillen,
``Rational homotopy theory",\hsm
\textit{Ann.\ Math.\ (2)} \textbf{90} (1969), pp.\ 205-295.
%
\bibitem[R]{RezkM}
C.~Rezk,
``A model for the homotopy theory of homotopy theory'',\hsm
\textit{Trans.\ AMS} \textbf{353} (2001), pp.~973-1007.
%
\bibitem[RV]{RVeriIC}
E.~Riehl \& D.~Verity,
\textit{Elements of {$\infty$}-category theory},\hsm
Camb.\ U.\ Press, Cambridge, 2022.
%
\bibitem[Sa]{SagaU}
S.~Sagave,
``Universal Toda brackets of ring spectra'',\hsm
\textit{Trans.\ AMS} \textbf{360} (2008), pp.\ 2767-2808.
%
\bibitem[Sc1]{SchweSU}
S.~Schwede,
``The stable homotopy category has a unique model at the prime $2$'',\hsm
\textit{Adv.\ Math.} \textbf{164} (2001), pp.~24-40.
%
\bibitem[Sc2]{SchweSR}
S.~Schwede,
``The stable homotopy category is rigid'',
\textit{Ann.\ Math.\ (2)} \textbf{166} 2007), pp.~837-863.
%
\bibitem[Se]{SegCC}
G.B. Segal,
``Categories and cohomology theories'',\hsm
\textit{Topology} \textbf{13} (1974), 293-312.
%
\bibitem[Sh]{ShipA}
B.E.~Shipley,
``An algebraic model for rational $S\sp{1}$-equivariant stable homotopy theory'',\hsm
\textit{Q.\ J.\ Math.} \textbf{53} 2002), pp.~87-110.
%
\bibitem[Sp]{SpanS}
E.H.~Spanier,
``Secondary operations on mappings and cohomology'',\hsm
\textit{Ann.\ Math.\ (2)} \textbf{75} (1962) No.\ 2, pp.\ 260-282.
%
\bibitem[St]{StoV}
C.R.~Stover,
``A Van Kampen spectral sequence for higher homotopy groups'',\hsm
\textit{Topology} \textbf{29} (1990), pp.~9-26.
%
\bibitem[T1]{TodG}
H.~Toda,
``Generalized {Whitehead} products and homotopy groups of spheres'',\hsm
\textit{J.\ Inst.\ Polytech.\ Osaka City U., Ser.\ A, Math.} \textbf{3}
(1952), pp.\ 43-82.
%
\bibitem[T2]{TodC}
H.~Toda,
\textit{Composition methods in the homotopy groups of spheres},\hsm
Adv.\ in Math.\ Study \textbf{49}, Princeton U. Press, Princeton, 1962.
%
\bibitem[T]{ToenA}
B.~To{\"{e}}n,
``Vers une Axiomatisation de la Th{\'{e}}orie des Cat{\'{e}}gories Sup{\'{e}}rieures'',
\textit{$K$-Theory} \textbf{34} (2005), pp.~233-263.
%
\bibitem[Wal]{GWalkL}
G.\ Walker,
``Long Toda brackets'',\hsm
in \textit{Proc.\ Adv.\ Studies Inst.\ on Algebraic Topology, vol.\ III},
Aarhus U.\ Mat.\ Inst.\ Various Publ.\ Ser.\ \textbf{13}, Aarhus 1970,
pp.\ 612-631.
%
\bibitem[WX]{WXuT}
G.~Wang \& Z.~Xu,
``The triviality of the $61$-stem in the stable homotopy groups of spheres",
\textit{Ann.\ Math.\ (2)} \textbf{186} (2017), pp.~501-580.
%
\bibitem[Wan]{HCWangHF}
H.-C.~Wang,
``The homology groups of the fibre bundles over the sphere'',\hsm
\textit{Duke Math.~J.} \textbf{16} (1949), pp.~33-38.
%
\bibitem[Wh]{GWhE}
G.W.~Whitehead,
\textit{Elements of homotopy theory},\hsm
Springer-Verlag, Berlin-\-New York, 1971.
%
\end{thebibliography}
\end{document}